\documentclass[11pt,reqno]{amsart}

\usepackage{amsmath,amssymb,mathrsfs}
\usepackage{graphicx,cite,times}
\usepackage{hyperref}
\hypersetup{
    colorlinks=true,
    linkcolor=blue,
    filecolor=gray,
    urlcolor=blue,
    citecolor=blue,
}

\usepackage[doipre={doi:~}]{uri}
\usepackage{cases}
 \usepackage{dsfont}
 \usepackage[mathscr]{eucal}
 \usepackage{color}
\newtheorem{theo}{Theorem}[section]

\newtheorem{lemm}[theo]{Lemma}

\newtheorem{defi}[theo]{Definition}

\renewcommand{\Re}{\operatorname{Re}}

\numberwithin{equation}{section}
\newcommand{\mycase}[2]{%
    \noindent {#1:} #2\par  
}

\begin{document}

\title[Subwavelength resonances in 2D elastic media with high contrast]{
Subwavelength resonances in two-dimensional elastic media with high contrast}

\author{Yuanchun Ren}
\address{School of Mathematics and Statistics, Center for Mathematics and Interdisciplinary Sciences, Northeast Normal University, Changchun, Jilin 130024, P.R.China.}
\email{renyuanchun@nenu.edu.cn}

\author{Yixian Gao}

\address{School of
Mathematics and Statistics, Center for Mathematics and
Interdisciplinary Sciences, Northeast Normal University, Changchun, Jilin 130024, P.R.China. }
\email{gaoyx643@nenu.edu.cn}

\thanks{The work of YG is supported by NSFC grants (project numbers, 12371187 and 12071065) and Science and Technology Development Plan Project of Jilin Province(project number, 20240101006JJ).
}



\begin{abstract}
This paper employs layer potential techniques to investigate wave scattering in two-dimensional elastic media exhibiting high contrasts in both Lam\'{e} parameters and density. Our contributions are fourfold. First, we construct an invertible operator based on the kernel spaces of boundary integral operators, which enables the characterization of resonant frequencies through an orthogonality condition. Second, we use asymptotic analysis to derive the equation governing the leading-order terms of these resonant frequencies. Third, we analyze the scattered field in the interior domain for incident frequencies across different regimes and characterize the longitudinal and transverse far-field patterns in the exterior domain. Finally, we examine the subwavelength bandgap in the phononic crystal with a dilute structure.

\end{abstract}

\maketitle


\section{Introduction}\label{sec:1}

Metamaterials are a class of artificial composite materials, typically comprising periodic or random arrays of subwavelength resonators \cite{Liu2005Analytic, Matlack2016Composite}. These resonators are characterized by high contrasts between their specific physical parameters and those of the background medium. A ``resonator" is a device that amplifies waves via resonance at specific frequencies, a concept applicable to acoustic, electromagnetic, and elastic waves. The ``subwavelength" means the incident wavelength is significantly larger than the resonator's size.
Metamaterials enable a range of exceptional applications, including super-resolution imaging \cite{Casse2010Super, Blasten2020Localization}, super-focusing \cite{AFGLZ2017, Lanoy2015Subwavelength}, waveguiding \cite{Alu2004Guided, AHY2021}, metasurfaces \cite{Ammari2017screens, Ammari2019point}, and invisibility cloaking \cite{Milton2006cloaking, Ammari2013Spectral, Ando2017Spectrum}. Due to these promising applications and feasible fabrication methods, the study of subwavelength resonances induced by high-contrast physical parameters has become a major research focus in recent decades.

In acoustics, bubbles embedded in a liquid can exhibit Minnaert resonance \cite{Ammari2018Minnaert,Feppon2022modal}, 
a phenomenon resulting from the high density contrast between the gas and liquid.
 To improve bubble stability, an effective strategy is to replace the liquid with a soft elastic material featuring a small shear modulus \cite{Calvo2012Low,Chen2024Resonant}. 
 Similarly, in electromagnetics, dielectric particles with high refractive indices function as subwavelength resonators \cite{AAMYB2016,Libowenmathematical}. 
 For elastic waves, both physical experiments and numerical simulations have confirmed that the lead ball coated with silicone rubber, characterized by high contrasts in Lam\'{e} parameters and density, exhibits localized resonant behavior \cite{LZMZYCS2000,Liu2005Analytic}.
 The authors in \cite{Chen2025Analysis} and \cite{LZ2024} independently utilized 
 a variational method and layer potential techniques to develop 
 a rigorous mathematical framework for elastic wave scattering
  and subwavelength resonant phenomena induced by high-contrast material parameters.

The aforementioned references predominantly address the three-dimensional (3D) case. 
In this setting, the static single layer potential operators--denoted by $\mathcal{S}_D$ for the Laplace equation and $\boldsymbol{\mathcal{S}}_D$ for the Lam\'{e} system--are always invertible \cite{Lay,chang2007spectral}. Moreover, these operators and their inverses are analytic with respect to the wavenumber $k$ in the low-frequency regime.

In contrast, the two-dimensional (2D) case introduces distinct mathematical challenges, which manifest in three key aspects:
\begin{enumerate}
    \item The operators $\mathcal{S}_D$ and $\boldsymbol{\mathcal{S}}_D$ may fail to be invertible. There exist bounded domains $D$ for which $\boldsymbol{\mathcal{S}}_D[\boldsymbol{\varphi}] = \boldsymbol 0$ for some non-trivial $\boldsymbol{\varphi}$. Notably, even for a disk, the kernel of $\boldsymbol{\mathcal{S}}_D$ depends on its radius $R$ \cite{ando2018spectral}. A similar non-invertibility phenomenon for $\mathcal{S}_D$ is also known \cite{Verchota1984Layer}.

    \item The asymptotic expansions of the fundamental solutions in the low-frequency regime differ significantly between the 2D and 3D cases. Specifically, the 3D expansion is a power series in $k^j$, while its 2D counterpart involves a double series comprising terms of the form $k^{2j}$ and $k^{2j} \ln k$ for $j \in \mathbb{N}$.

    \item  The leading-order term in the asymptotic expansion of the single layer potential operator is not simply the static operator. It also incorporates a logarithmic singularity $\ln k$ and an imaginary component. We denote this leading-order term as $\hat{\mathcal{S}}_D^k$ and $\hat{\boldsymbol{\mathcal{S}}}_{D}^k$ for the Helmholtz equation and the Lam\'{e} system, respectively.
\end{enumerate}

This complexity  presents significant challenges for the mathematical investigation of subwavelength resonances in $\mathbb{R}^2$. Notably, the equation governing the subwavelength resonant frequencies of 2D Minnaert bubbles is detailed in the appendix of \cite{Ammari2018Minnaert}.
Compared with acoustic waves (exclusively longitudinal wave), the investigation of elastic waves faces heightened complexity, attributable to the coupling of longitudinal and transverse waves as well as the vectorial characteristics of the Lam\'{e} operator. Furthermore, while the space spanned by all linear solutions to the Laplace equation with Neumann boundary conditions is one-dimensional, the corresponding space for the static Lam\'{e} system with Neumann boundary conditions is three-dimensional.
These  factors not only affect the invertibility of $\hat{\boldsymbol{\mathcal{S}}}^k_{D}$ but also lead to the leading-order terms of the resonant frequencies  are governed by  two $3\times 3$ matrices. Consequently, to rigorously establish the relationship between resonant frequencies and high-contrast parameters, we restrict our focus in this work to the elastic scattering problem for the resonator with radial geometry.

This paper extends the three-dimensional (3D) high-contrast elastic scattering analysis of \cite{Chen2025Analysis} and \cite{LZ2024} to the more challenging two-dimensional (2D) case. Unlike earlier approaches that employed primal and dual variational principles for the elastostatic system to identify resonant conditions \cite{Li2016anomalous, Li2017three}, the works of \cite{Chen2025Analysis} and \cite{Ren2025Subwavelength} derived explicit expressions for subwavelength resonant frequencies by introducing the Dirichlet-to-Neumann (DtN) map and an auxiliary sesquilinear form. However, this variational approach is inapplicable in 2D due to the lack of analyticity of the DtN map with respect to the wavenumber $k$. To overcome this challenge, the present study adopts the boundary integral equation method, leveraging the spectral properties of the Neumann-Poincar\'{e} (NP) operator. The main contributions of this work are summarized as follows:

\begin{enumerate}
    \item \textbf{Asymptotic Analysis of Integral Operators:} Building on the asymptotic expansion of the zeroth-order Hankel function of the first kind, we establish the low-frequency asymptotic expressions for the relevant boundary integral operators with respect to the wavenumber $k$. Furthermore, we prove the invertibility of the leading-order operator $\hat{\boldsymbol{\mathcal{S}}}^k_{D}$. In contrast to \cite{LZ2024}, our definitions of fundamental solutions and associated operators are parameterized directly by $k$ rather than by frequency. This choice leads to inherent differences in the analysis of both the boundary integral operators for the elastic scattering problem and the coefficients of the scattered field.

    \item \textbf{Reformulation via Boundary Integral Equations and Kernel Analysis:} The coupled internal-external scattering problem is reformulated as a system of boundary integral equations. The kernel space of the leading-order boundary integral operator is characterized by the spectral properties of $\boldsymbol{\mathcal{S}}_D$ and the NP operator. This kernel space bifurcates into two distinct cases, contingent on the invertibility of $\boldsymbol{\mathcal{S}}_D$. We also determine the kernel space of the adjoint operator. Whereas \cite{LZ2024} exploited the invertibility of $\boldsymbol{\mathcal{S}}_D$ and the linear independence of its action on an orthonormal basis to obtain the corresponding kernel spaces for an arbitrary 3D Lipschitz domain, the 2D case is fundamentally different. We introduce an invertible operator associated with these kernel spaces and employ asymptotic analysis to derive the equations governing the leading-order terms of the resonant frequencies. In 3D, these terms are explicitly given by the eigenvalues of a matrix, which is inherently symmetric. By contrast, in 2D, explicit expressions remain elusive even for a disk; we instead obtain governing equations involving two diagonal matrices.

    \item \textbf{Analysis of the Scattered Field:} We explicitly derive the scattered field $\boldsymbol u_D$ inside the resonator $D$ for incident frequencies across different regimes. When the incident frequency $\omega$ approaches a subwavelength resonant frequency such that $\omega^2\ln \omega=\mathcal{O}(\epsilon)$, the resonant enhancement coefficients of $\boldsymbol u_D$ are of order $\mathcal{O}(|\ln \omega|)$, where $\epsilon$ denotes the density contrast between the interior and exterior elastic materials. Additionally, we characterize the far-field patterns of the scattered field in the exterior domain, encompassing both longitudinal and transverse components.

    \item \textbf{Subwavelength Bandgaps in Dilute Phononic Crystals:} We build upon the established existence of subwavelength bandgaps in 2D phononic crystals \cite{Ren2025Subwavelength} to determine their spectral position in a dilute structure. This analysis exploits the small size of the resonator and the invertibility of $\hat{\boldsymbol{\mathcal{S}}}^k_{D}$. In this dilute regime, where the resonator size is much smaller than the inter-resonator distance, inter-particle interactions can be neglected.
\end{enumerate}

This paper is structured into five main sections. Section \ref{sec:2} presents the mathematical formulation of the elastic scattering problem in $\mathbb{R}^2$, introducing the definitions and fundamental properties of the relevant boundary integral operators. In Section \ref{sec:3}, we derive the equations that govern the leading-order terms of the resonant frequencies. Section \ref{sec:4} analyzes the asymptotic behavior of the scattered field inside the resonator and the corresponding far-field patterns when the incident frequency lies in different regimes. Section \ref{sec:5} examines the positioning of the subwavelength bandgap in a dilute-structure phononic crystal. Finally, Section \ref{sec:6} provides concluding remarks and outlines future research directions.


\section{Mathematical setup}\label{sec:2}
This section is devoted to the mathematical characterization of two-dimensional elastic wave scattering. We also present the definitions and essential properties of layer potential operators to establish the foundation for the subsequent analysis.

\subsection{Problem formulation}
Consider the two-dimensional elastic wave scattering by the scatterer $D$, where $D$ is a disk with radius $R$.
Let the displacement field $\boldsymbol{u}=(u_i(\boldsymbol{x}))_{i=1}^2$ satisfy the time-harmonic Lam\'{e} system:
\begin{align*}
    \mathcal{L}^{\widehat{\lambda},\widehat{\mu}} \boldsymbol{u}(\boldsymbol{x}) + \widehat{\rho} \, \omega^2 \boldsymbol{u}(\boldsymbol{x}) = \boldsymbol 0, \quad \boldsymbol{x} \in \mathbb{R}^2,
\end{align*}
where the partial differential operator $\mathcal{L}^{\widehat{\lambda},\widehat{\mu}}$ is defined by
\begin{align*}
    \mathcal{L}^{\widehat{\lambda},\widehat{\mu}} \boldsymbol{u} := \widehat{\mu} \, \nabla \cdot (\nabla \boldsymbol{u} + \nabla \boldsymbol{u}^\top) + \widehat{\lambda} \, \nabla (\nabla \cdot \boldsymbol{u})
\end{align*}
with $\nabla \boldsymbol{u} = (\partial_{x_j} u_i)_{i,j=1}^2$ denoting the displacement gradient matrix and ${}^\top$ denoting the matrix transpose. Define the conormal derivative $\partial_{\boldsymbol{\nu};\widehat{\lambda},\widehat{\mu}}$ as
\begin{align}\label{normal}
\partial_{\boldsymbol{\nu};\widehat{\lambda},\widehat{\mu}}\boldsymbol{u}:=\widehat{\lambda}(\nabla \cdot \boldsymbol{u})\boldsymbol{\nu}+\widehat{\mu}(\nabla \boldsymbol{u}+\nabla \boldsymbol{u}^\top)\boldsymbol{\nu}
\end{align}
with $\boldsymbol{\nu}$ representing the unit outward normal vector to $\partial D$.
The material parameters--the density $\widehat{\rho}$ and the Lam\'{e} parameters $\widehat{\lambda}$ and $\widehat{\mu}$--are piecewise constant, defined by
\begin{align*}
    (\widehat{\rho}; \mathbb{R}^2) =(\rho; \mathbb{R}^2 \setminus \overline{D}) \cup (\widetilde{\rho}; D) , \quad 
    (\widehat{\lambda}, \widehat{\mu}; \mathbb{R}^2) = (\lambda, \mu; \mathbb{R}^2 \setminus \overline{D}) \cup (\widetilde{\lambda}, \widetilde{\mu}; D)  .
\end{align*}
These parameters satisfy the strong convexity conditions in both media
\begin{align*}
    \mu > 0, \quad \lambda + \mu > 0, \quad \widetilde{\mu} > 0, \quad \text{and} \quad \widetilde{\lambda} + \widetilde{\mu} > 0.
\end{align*}
The corresponding shear and compressional wave velocities in the exterior domain $\mathbb{R}^2 \setminus \overline{D}$ and the obstacle $D$ are given, respectively
\begin{align}
    \begin{cases}
        c_s = \sqrt{\mu / \rho}, \\
        c_p = \sqrt{(\lambda + 2\mu) / \rho},
    \end{cases}
    \quad \text{and} \quad
    \begin{cases}
        \widetilde{c}_s = \sqrt{\widetilde{\mu} / \widetilde{\rho}}, \\
        \widetilde{c}_p = \sqrt{(\widetilde{\lambda} + 2\widetilde{\mu}) / \widetilde{\rho}}.
    \end{cases}
    \label{velocity}
\end{align}

Consider the time-harmonic incident wave $\boldsymbol{u}^{\mathrm{in}}$ (with the frequency $\omega$) that is scattered by the scatter   $D$.
The resulting total field $\boldsymbol{u}$ is governed by  the following Lam\'{e} system:
\begin{align}\label{Lame0}
\left\{ \begin{aligned}
&\mathcal{L}^{\lambda,\mu}\boldsymbol{u}+\rho\omega^2\boldsymbol{u}= \boldsymbol 0&&\text{in }\mathbb{R}^2\backslash\overline{D},\\
&\mathcal{L}^{\widetilde{\lambda},\widetilde{\mu}}\boldsymbol{u}+\widetilde{\rho}\omega^2\boldsymbol{u}=\boldsymbol 0&&\text{in } D,\\
&\boldsymbol{u}|_+=\boldsymbol{u}|_-&&\text{on } \partial D,\\
&\partial_{\boldsymbol{\nu};\lambda,\mu}\boldsymbol{u} \big|_+=\partial_{\boldsymbol{\nu};\widetilde{\lambda},\widetilde{\mu}}\boldsymbol{u}\big|_-&&\text{on }\partial D,
\end{aligned}\right.
\end{align}
where the notations $|_{+}$ and $|_{-}$ denote  the traces on $\partial D$ taken from outside and inside of $D$, respectively. 
To ensure that the scattering wave propagates outward, the following  Kupradze radiation conditions are imposed at infinity \cite{Lay}:
\begin{align*}
\left\{ \begin{aligned}
&(\nabla\times\nabla\times(\boldsymbol{u}-\boldsymbol{u}^{\mathrm{in}}))(\boldsymbol x)\times\frac{\boldsymbol x}{|\boldsymbol x|}-\frac{\mathrm{i}\omega}{c_s}\nabla\times(\boldsymbol{u}-\boldsymbol{u}^{\mathrm{in}})(\boldsymbol x)=\mathcal{O}\left(|\boldsymbol x|^{-\frac{3}{2}}\right)&& \text{as }|\boldsymbol x|\rightarrow+\infty,\\
&\frac{\boldsymbol x}{|\boldsymbol x|}\cdot (\nabla(\nabla\cdot(\boldsymbol{u}-\boldsymbol{u}^{\mathrm{in}})))(\boldsymbol x)-\frac{\mathrm{i}\omega}{c_p}\nabla\cdot(\boldsymbol{u}-\boldsymbol{u}^{\mathrm{in}})(\boldsymbol x)=\mathcal{O}\left(|\boldsymbol x|^{-\frac{3}{2}}\right)&&\text{as }|\boldsymbol x|\rightarrow+\infty.
\end{aligned}\right.
\end{align*}

To quantify the discrepancy between the material parameters of the scatterer $D$ and the background medium, we introduce the following dimensionless contrast parameters:
\begin{align}\label{ratio}
(\widetilde{\lambda},\widetilde{\mu})=\frac{1}{\delta}(\lambda,\mu), \quad\widetilde{\rho}=\frac{1}{\epsilon}\rho, \quad\tau=\frac{c_s}{\widetilde{c}_s}=\frac{c_p}{\widetilde{c}_p}.
\end{align}
In view of the wave speeds defined in \eqref{velocity}, the parameter $\tau$ satisfies the relation
\begin{align}\label{contrast}
\tau=\sqrt{\frac{\delta}{\epsilon}}.
\end{align}
Using the relations in \eqref{ratio} and \eqref{contrast}, the system \eqref{Lame0} can be reformulated equivalently as
\begin{align}\label{Lame}
\left\{ \begin{aligned}
&\mathcal{L}^{\lambda,\mu}\boldsymbol{u}+\rho\omega^2\boldsymbol{u}=\boldsymbol 0&&\text{in }\mathbb{R}^2\backslash\overline{D},\\
&\mathcal{L}^{\lambda,\mu}\boldsymbol{u}+\rho\tau^2\omega^2\boldsymbol{u}= \boldsymbol 0&&\text{in } D,\\
&\boldsymbol{u}|_+=\boldsymbol{u}|_-&&\text{on } \partial D,\\
&\delta\partial_{ \boldsymbol{\nu}} \boldsymbol{u}\big|_+=\partial_{\boldsymbol{\nu}}\boldsymbol{u}\big|_-&&\text{on }\partial D.
\end{aligned}\right.
\end{align}
Here, we abbreviate $\partial_{ \boldsymbol{\nu}}\boldsymbol{u}=\partial_{\boldsymbol{\nu};\lambda,\mu}\boldsymbol{u}$, with $\partial_{\boldsymbol{\nu};\lambda,\mu}\boldsymbol{u}$ given by \eqref{normal}.

The subsequent analysis relies on the following physical assumptions:
\begin{itemize}
\item{
High contrast: the scatterer's Lam\'{e} parameters and density differ significantly from the background ($\delta, \epsilon \to 0^+$).}

\item{Matched wave speed: the wave propagation velocities inside and outside are of the same order ($\tau = \mathcal{O}(1)$).}

\item {Subwavelength regime: the exciting frequency is small ($\omega \to 0$).}
\end{itemize}
These are formalized by the asymptotic conditions:
\begin{align*}
\tau = \mathcal{O}(1), \quad \delta \to 0^{+}, \quad \epsilon \to 0^{+}, \quad \omega \to 0.
\end{align*}

\subsection{Layer potential operators}

In this subsection, we present the definitions of layer potential operators and derive the asymptotic expansions of these operators in the low-frequency regime.

Define the single layer potential $\boldsymbol{\widetilde{\mathcal{S}}}^{k}_{D}: L^2(\partial D)^2\rightarrow H^{3/2}_{\mathrm {loc}}(\mathbb{R}^2)^2$ as
\begin{align*}
\boldsymbol{\widetilde{\mathcal{S}}}^{k}_{D}[\boldsymbol\varphi](\boldsymbol x):=\int_{\partial D}\boldsymbol{G}^{k}(\boldsymbol x-\boldsymbol y)\boldsymbol\varphi(\boldsymbol y)\mathrm{d}\sigma(\boldsymbol y),\quad \boldsymbol x\in\mathbb{R}^2\backslash\partial D.
\end{align*}
where the  Green's tensor $\boldsymbol{G}^{k}(\boldsymbol x-\boldsymbol y)$ is the fundamental solution to
\begin{align*}
(\mathcal{L}^{\lambda,\mu}+k^2)\boldsymbol{G}^{k}(\boldsymbol x-\boldsymbol y)=\boldsymbol{\delta}(\boldsymbol x-\boldsymbol y)\boldsymbol{I}.
\end{align*}
Here $\boldsymbol{\delta}$ denotes the Dirac delta distribution   and $\boldsymbol{I}$ is  the $2\times 2$ identity matrix.
Explicitly,  the matrix  fundamental solution  $\boldsymbol{G}^{k}(\boldsymbol x)=(G_{ij}^{k}(\boldsymbol x))^2_{i,j=1}$ is given by
\begin{equation}\label{fun solu}
\begin{aligned}
G_{ij}^{k}(\boldsymbol x)=\begin{cases}\frac{1}{4\pi}\left(\frac{1}{\mu}+\frac{1}{\lambda+2\mu}\right)\delta_{ij}\ln|\boldsymbol x|-\frac{1}{4\pi}\left(\frac{1}{\mu}-\frac{1}{\lambda+2\mu}\right)\frac{x_{i}x_{j}}{|\boldsymbol x|^{2}},&k=0,\\
-\frac{\mathrm{i}}{4\mu}\delta_{ij}H_{0}^{(1)}\left(\frac{k|\boldsymbol x|}{\sqrt{\mu}}\right)+\frac{\mathrm{i}}{4k^{2}}\partial_{i}\partial_{j}
 \left(H_{0}^{(1)}\left(\frac{k|\boldsymbol x|}{\sqrt{\lambda+2\mu}}\right)-H_{0}^{(1)}\left(\frac{k|\boldsymbol x|}{\sqrt{\mu}}\right)\right),&k\neq 0,
 \end{cases}
\end{aligned}
\end{equation}
where $H_{0}^{(1)}$ denotes the zeroth-order Hankel function of the first kind and $x_i$ is the $i$-th component of the vector $\boldsymbol x$.

We further define the boundary integral operators $\boldsymbol{\mathcal{S}}^{k}_{D}: L^2(\partial D)^2\rightarrow H^{1}(\partial D)^2$ and $\boldsymbol{\mathcal{K}}^{k}_{D}, \boldsymbol{\mathcal{K}}^{k,*}_{D}: L^2(\partial D)^2\rightarrow L^2(\partial D)^2$  as
\begin{align*}
\boldsymbol{\mathcal{S}}^{k}_{D}[\boldsymbol\varphi](\boldsymbol x)&:=\int_{\partial D}\boldsymbol G^{k}(\boldsymbol x-\boldsymbol y)\boldsymbol\varphi(\boldsymbol y)\mathrm{d}\sigma(\boldsymbol y),\quad \boldsymbol x\in\partial{D},\\
\boldsymbol{\mathcal{K}}^{k}_{D}[\boldsymbol\varphi](\boldsymbol x)&:=\mathrm{p.v.}\int_{\partial D} \partial_{\boldsymbol{\nu}_{\boldsymbol y}}\overline{\boldsymbol G^{k}(\boldsymbol x-\boldsymbol y)}\boldsymbol\varphi(\boldsymbol y)\mathrm{d}\sigma(\boldsymbol y),\quad \boldsymbol x\in\partial{D},\\
\boldsymbol{\mathcal{K}}^{k,*}_{D}[\boldsymbol\varphi](\boldsymbol x)&:=\mathrm{p.v.}\int_{\partial D} \partial_{\boldsymbol{\nu}_{\boldsymbol x}}\boldsymbol G^{k}(\boldsymbol x-\boldsymbol y)\boldsymbol\varphi(\boldsymbol y)\mathrm{d}\sigma(\boldsymbol y),\quad \boldsymbol x\in\partial{D},
\end{align*}
where $\mathrm{p.v.}$ denotes the Cauchy principal value. The operator $\boldsymbol{\mathcal{K}}^{k,*}_{D}$ is referred to as Neumann-Poincar\'{e} operator and $\boldsymbol{\mathcal{K}}^{k,*}_{D}$ is the $L^2$-adjoint of $\boldsymbol{\mathcal{K}}^{k}_{D}$. Moreover, the following jump relations hold:
\begin{align}\label{Jump2}
\partial_{\boldsymbol{\nu}}\boldsymbol{\mathcal{S}}^{k}_{D}[\boldsymbol\varphi]\big|_{\pm}(\boldsymbol x)=\left(\pm\frac{1}{2}\boldsymbol{\mathcal{I}}+\boldsymbol{\mathcal{K}}^{k,*}_{ D}\right)[\boldsymbol\varphi](\boldsymbol x),\quad \boldsymbol x\in\partial{D},
\end{align}
with the identity operator $\boldsymbol{\mathcal{I}}$.
Subsequently, the elastostatic operators $\boldsymbol{\widetilde{\mathcal{S}}}^{0}_{D}$, $\boldsymbol{\mathcal{S}}^{0}_{D}$, $\boldsymbol{\mathcal{K}}^{0}_{D}$ and $\boldsymbol{\mathcal{K}}^{0,*}_{D}$ will be simplified to $\boldsymbol{\widetilde{\mathcal{S}}}_{D}$, $\boldsymbol{\mathcal{S}}_{D}$, $\boldsymbol{\mathcal{K}}_{D}$ and $\boldsymbol{\mathcal{K}}^{*}_{D}$, respectively.

Introduce four parameters
\begin{align}\label{parameters}
\tau_1&=\frac{1}{2}\left(\frac{1}{\mu}+\frac{1}{\lambda+2\mu}\right),&&\tau_2=\frac{1}{2}\left(\frac{1}{\mu}-\frac{1}{\lambda+2\mu}\right),\nonumber\\
\varrho_1&=\frac{1}{8\pi}\left(\gamma-\ln 2-\frac{\mathrm{i}\pi}{2}-1\right),&&\varrho_2=-\frac{1}{128\pi}\left(\gamma-\ln2-\frac{\mathrm{i}\pi}{2}-\frac{3}{2}\right),
\end{align}
with the Euler's constant $\gamma$.

Next, we derive the asymptotic expansions of $\boldsymbol{\mathcal{S}}^{k}_{D}$ and $\boldsymbol{\mathcal{K}}^{k,*}_{D}$ with respect to $k$ as $k\rightarrow 0$, based on the asymptotic expansion of the fundamental solution $\boldsymbol G^k$.

\begin{lemm}\label{le:series}
As $k\rightarrow 0$, the fundamental solution $\boldsymbol G^k$ admits the asymptotic expansion
\begin{align*}
\boldsymbol G^{k}(\boldsymbol x)=&\boldsymbol G^{0}(\boldsymbol x)+ \beta_k \boldsymbol I + k^2 \ln k \boldsymbol A(\boldsymbol x)+k^2 \boldsymbol B(\boldsymbol x)+\mathcal{O}(k^4 \ln k+k^4),
\end{align*}
where  $\boldsymbol A(\boldsymbol x)=(a_{ij}(\boldsymbol x))^2_{i,j=1},$ $\boldsymbol B(\boldsymbol x)=(b_{ij}(\boldsymbol x))^2_{i,j=1},$ and
\begin{align}\label{beta}
\beta_k = \frac{\tau_1}{4\pi}(2\gamma - \pi \mathrm{i} - \ln 4) + \frac{\tau_2}{4\pi} + \frac{\ln (k/\sqrt{\lambda+2\mu})}{4\pi(\lambda + 2\mu)} + \frac{\ln (k/\sqrt{\mu})}{4\pi \mu}.
\end{align}
Here, the matrix elements are given by
\begin{align*}
&a_{ij}(\boldsymbol x)=\delta_{ij}\left(-\frac{3|\boldsymbol x|^2}{32\pi\mu^2}-\frac{|\boldsymbol x|^2}{32\pi(\lambda+2\mu)^2}\right)-\frac{x_ix_j}{16\pi(\lambda+2\mu)^2}+\frac{x_ix_j}{16\pi\mu^2},\\
&b_{ij}(\boldsymbol x)=\delta_{ij}\left(\sigma_1|\boldsymbol x|^2+\left(\tau_1 \tau_2-\mu^{-2}\right)\frac{|\boldsymbol x|^2\ln|\boldsymbol x|}{8\pi}\right)+\sigma_2x_ix_j+\frac{\tau_1\tau_2}{4\pi}x_ix_j\ln|\boldsymbol x|
\end{align*}
with the constants  given by 
\begin{align*}
&\sigma_1=\frac{1}{\mu^2}\left(\varrho_1-\frac{5\ln\mu}{256\pi}+4\varrho_2+\frac{1}{128\pi}\right)-\frac{1}{(\lambda+2\mu)^2}\left(4\varrho_2+\frac{1}{128\pi}-\frac{\ln(\lambda+2\mu)}{64\pi}\right),\\
&\sigma_2=\frac{8}{\mu^2}\left(\varrho_2-\frac{\ln\mu}{256\pi}\right)-\frac{8}{(\lambda+2\mu)^2}\left(\varrho_2-\frac{\ln(\lambda+2\mu)}{256\pi}\right)+\frac{3\tau_1\tau_2}{16\pi}.
\end{align*}
\end{lemm}

\begin{proof}
Using the following behavior of the Hankel function $H_{0}^{(1)}$ near $0$ (cf.\cite[p.22]{Lay}):
\begin{align*}
-\frac{\mathrm{i}}{4}H_{0}^{(1)}(k|\boldsymbol x|)=&\frac{1}{2\pi}\ln|\boldsymbol x|+\frac{1}{2\pi}(\ln k+\gamma-\ln 2)-\frac{\mathrm{i}}{4}+\left(-\frac{1}{8\pi}\ln(k|\boldsymbol x|)+\varrho_1\right)(k|\boldsymbol x|)^2\\
&+\left(\frac{1}{128\pi}\ln(k|\boldsymbol x|)+\varrho_2\right)(k|\boldsymbol x|)^4+\mathcal{O}(k^6 \ln k+k^6), \quad\text{for }k\ll 1.
\end{align*}
By substituting this expansion into the expression for the fundamental solution $\boldsymbol{G}^k$ given in \eqref{fun solu} and performing careful term-by-term computation, we can obtain the desired asymptotic expansion.

\end{proof}

\begin{lemm}\label{SK:asy}
As $k\rightarrow 0$,  the boundary integral operators $\boldsymbol{\mathcal{S}}^{k}_{D}: L^{2}(\partial D)^2\rightarrow H^{1}(\partial D)^2$ and $\boldsymbol{\mathcal{K}}^{k,*}_{D}:L^{2}(\partial D)^2\rightarrow L^{2}(\partial D)^2$ admit the following asymptotic expansions:
\begin{align*}
\boldsymbol{\mathcal{S}}^{k}_{D}=\hat{\boldsymbol{\mathcal{S}}}^k_{D}+k^2\ln k \boldsymbol{\mathcal{S}}^{(1)}_{D,1}+k^2\boldsymbol{\mathcal{S}}^{(2)}_{D,1}+\mathcal{O}(k^4\ln k+k^4),
\end{align*}
\begin{align*}
\boldsymbol{\mathcal{K}}^{k,*}_{D}=\boldsymbol{\mathcal{K}}^*_{D}+k^2\ln k \boldsymbol{\mathcal{K}}^{(1),*}_{D,1}+k^2\boldsymbol{\mathcal{K}}^{(2),*}_{D,1}+\mathcal{O}(k^4\ln k+k^4),
\end{align*}
where
\begin{align}\label{S:hat}
\hat{\boldsymbol{\mathcal{S}}}^k_{D}[\boldsymbol \varphi]=\boldsymbol{\mathcal{S}}_{D}[\boldsymbol \varphi]+\beta_k\int_{\partial D}\boldsymbol \varphi(\boldsymbol y)\mathrm{d}\sigma(\boldsymbol y),
\end{align}
$$\boldsymbol{\mathcal{S}}^{(1)}_{D,1}[\boldsymbol \varphi]=\int_{\partial D}\boldsymbol A(\boldsymbol x-\boldsymbol y)\boldsymbol \varphi(\boldsymbol y)\mathrm{d}\sigma(\boldsymbol y),\quad\boldsymbol{\mathcal{S}}^{(2)}_{D,1}[\boldsymbol \varphi]=\int_{\partial D}\boldsymbol B(\boldsymbol x-\boldsymbol y)\boldsymbol \varphi(\boldsymbol y)\mathrm{d}\sigma(\boldsymbol y),$$
and
$$\boldsymbol{\mathcal{K}}^{(1),*}_{D,1}[\boldsymbol \varphi]=\int_{\partial D}\partial_{\boldsymbol{\nu}_{\boldsymbol x}}\boldsymbol A(\boldsymbol x-\boldsymbol y)\boldsymbol \varphi(\boldsymbol y)\mathrm{d}\sigma(\boldsymbol y),\quad\boldsymbol{\mathcal{K}}^{(2),*}_{D,1}[\boldsymbol \varphi]=\int_{\partial D}\partial_{\boldsymbol{\nu}_{\boldsymbol x}}\boldsymbol B(\boldsymbol x-\boldsymbol y)\boldsymbol \varphi(\boldsymbol y)\mathrm{d}\sigma(\boldsymbol y)$$
with $\beta_k$, $\boldsymbol A$ and $\boldsymbol B$  defined in Lemma \ref{le:series}.
\end{lemm}

Consider the following homogeneous Neumann problem for the Lam\'{e} system
\begin{align*}
\begin{cases}
\mathcal{L}^{\lambda,\mu} \boldsymbol u = \boldsymbol 0 & \text{in } D, \\
\partial_{\boldsymbol \nu} \boldsymbol u = \boldsymbol 0 & \text{on } \partial D.
\end{cases}
\end{align*}
The solution space of this problem is spanned by the following three linearly independent functions
\begin{align}\label{f}
\boldsymbol f^{(1)} = \frac{1}{\sqrt{2\pi R}} \begin{pmatrix} 1 \\ 0 \end{pmatrix}, \quad
\boldsymbol f^{(2)} = \frac{1}{\sqrt{2\pi R}} \begin{pmatrix} 0 \\ 1 \end{pmatrix}, \quad
\boldsymbol f^{(3)} = \frac{1}{\sqrt{2\pi R^3}} \begin{pmatrix} x_2 \\ -x_1 \end{pmatrix},
\end{align}
where $x_i$ denotes the $i$-th component of $\boldsymbol x\in D$. Obviously, These basis functions are orthonormal in $L^2(\partial D)^2$, satisfying
\begin{align}
\langle\boldsymbol f^{(i)},\boldsymbol f^{(j)}\rangle=\delta_{ij}, \quad\text{for } i, j =1, 2, 3,
\end{align}
 where $\langle\cdot,\cdot\rangle$ denotes  the inner product   in $L^2(\partial D)^2$. Furthermore, we introduce the spaces
  \[\mathcal{H}=L^2(\partial D)^2\times L^2(\partial D)^2\] and
   \[\mathcal{H}_1=H^1(\partial D)^2\times L^2(\partial D)^2.\]
Define a subspace of $L^2(\partial D)^2$ by
\[\mathbf{L}^2_{\boldsymbol f}(\partial D)=\{\boldsymbol h\in L^2(\partial D)^2:\langle\boldsymbol h,\boldsymbol f^{(i)}\rangle=0, \forall i=1,2,3\}.\]

\begin{lemm}\label{SK:ker}
(see \cite{ando2018spectral}) One has
\begin{itemize}
\item[(i)] The operator $\boldsymbol{\mathcal{S}}_{D}$ is invertible on the subspace  $\mathbf{L}^2_{\boldsymbol f}(\partial D)$.
\item[(ii)] The eigenvectors $\boldsymbol f^{(1)},\boldsymbol f^{(2)}$ and $\boldsymbol f^{(3)}$ span the eigenspace for the operator $\boldsymbol{\mathcal{K}}_{D}$ on $ L^{2}(\partial D)^2$ corresponding to the eigenvalue $\frac{1}{2}$.
\item[(iii)] When $D$ is a disk, the eigenvectors $\boldsymbol f^{(1)},\boldsymbol f^{(2)}$ and $\boldsymbol f^{(3)}$ span  the eigenspace for the  operator $\boldsymbol{\mathcal{K}}^*_{D}$ on $ L^{2}(\partial D)^2$ corresponding to the eigenvalue $\frac{1}{2}$.
\end{itemize}
\end{lemm}

\begin{lemm}\label{S:properties}
Let $D$ be a disk of radius $R$. Then the operator $\widetilde{\boldsymbol{\mathcal{S}}}_{D}$ satisfies the following identities:
\begin{itemize}
\item[(i)] $\widetilde{\boldsymbol{\mathcal{S}}}_{D}[\boldsymbol f^{(1)}]=\left(\tau_1 R\ln R-\frac{\tau_2}{2} R\right)\boldsymbol f^{(1)},$
\item[(ii)]$\widetilde{\boldsymbol{\mathcal{S}}}_{D}[\boldsymbol f^{(2)}]=\left(\tau_1 R\ln R-\frac{\tau_2}{2} R\right)\boldsymbol f^{(2)},$
\item[(iii)] $\widetilde{\boldsymbol{\mathcal{S}}}_{D}[\boldsymbol f^{(3)}]=-\frac{\tau_1}{2}R\boldsymbol f^{(3)},$
\end{itemize}
where $\tau_1$ and $\tau_2$ are constants defined in \eqref{parameters}.

\end{lemm}

\begin{proof}
We begin with the following identities from \cite{Ammari2013Spectral}
\begin{align}\label{disk}
\int_{\partial D}\frac{1}{2\pi}\ln|\boldsymbol x-\boldsymbol y|\mathrm{d}\sigma(\boldsymbol y)=R\ln R \quad\text{and}\quad\int_{\partial D}\frac{1}{2\pi}\ln|\boldsymbol x-\boldsymbol y|y_i\mathrm{d}\sigma(\boldsymbol y)=-\frac{R x_i}{2}, \quad\text{for } \boldsymbol x\in D.
\end{align}
For $i = 1,2$, we  can obtain
\begin{align*}
\int_{\partial D} \frac{1}{2\pi} \frac{(x_i-y_i)^2}{|\boldsymbol x-\boldsymbol y|^2}\mathrm{d}\sigma(\boldsymbol y)&=x_i \int_{\partial D} \frac{1}{2\pi} \frac{x_i-y_i}{|\boldsymbol x-\boldsymbol y|^2}\mathrm{d}\sigma(\boldsymbol y) - \int_{\partial D} \frac{1}{2\pi} \frac{(x_i-y_i)y_i}{|\boldsymbol x-\boldsymbol y|^2}\mathrm{d}\sigma(\boldsymbol y)\\
&=x_i\partial_{x_i}\int_{\partial D}\frac{1}{2\pi}\ln|\boldsymbol x-\boldsymbol y|\mathrm{d}\sigma(\boldsymbol y)-\partial_{x_i}\int_{\partial D}\frac{1}{2\pi}\ln|\boldsymbol x-\boldsymbol y|y_i\mathrm{d}\sigma(\boldsymbol y)\\
&=\frac{R}{2},
\end{align*}
and
\begin{align}\label{x12}
&\int_{\partial D} \frac{1}{2\pi} \frac{(x_1-y_1)(x_2-y_2)}{|\boldsymbol x-\boldsymbol y|^2}\mathrm{d}\sigma(\boldsymbol y)\nonumber\\
=&x_2 \partial_{x_1}\int_{\partial D} \frac{1}{2\pi} \ln |\boldsymbol x-\boldsymbol y|\mathrm{d}\sigma(\boldsymbol y) - \partial_{x_1}\int_{\partial D} \frac{1}{2\pi}\ln|\boldsymbol x-\boldsymbol y|y_2 \mathrm{d}\sigma(\boldsymbol y),\nonumber\\
=&0.
\end{align}
Then, (i) and (ii) follows directly  from the definition of $\widetilde{\boldsymbol{\mathcal{S}}}_{D}$ and the above computations.

In addition, by the symmetry property of the disk, we have for $i=1,2$ that
\begin{align*}
\int_{\partial D}\frac{1}{2\pi}\frac{(x_1-y_1)(x_2-y_2)}{|\boldsymbol x-\boldsymbol y|^2}y_i\mathrm{d}\sigma(\boldsymbol y)=0
\quad \text{and}\quad
\int_{\partial D}\frac{1}{2\pi}\frac{(x_i-y_i)^2}{|\boldsymbol x-\boldsymbol y|^2}y_j\mathrm{d}\sigma(\boldsymbol y)=0, \quad\text{for }i\neq j.
\end{align*}
Combining these with \eqref{disk} and \eqref{x12} yields (iii).
\end{proof}

According to Lemma \ref{S:properties}, we distinguish two distinct cases:

\mycase{Case 1} {$\tau_1 R\ln R=\frac{\tau_2}{2} R.$}
\mycase{Case 2} {$\tau_1 R\ln R\neq\frac{\tau_2}{2} R.$}

In Case 1, one obtains  
\[\boldsymbol{\mathcal{S}_{D}}[\boldsymbol f^{(1)}]=\boldsymbol{\mathcal{S}}_{D}[\boldsymbol f^{(2)}]=\boldsymbol 0,\]  which implies that the operator 
$\boldsymbol{\mathcal{S}}_{D}$ is clearly not invertible. In Case 2,  the invertibility of  $\boldsymbol{\mathcal{S}}_{D}$ follows directly from property (i) in Lemma \ref{SK:ker}.


\section{Subwavelength  resonance}\label{sec:3}

This section utilizes boundary integral equations to reformulate the scattering problem \eqref{Lame}. Based on the kernel space of the leading term operator, we
construct an auxiliary invertible operator. The subwavelength resonant frequencies are subsequently  derived through the  asymptotic analysis, which are closely linked to the diagonal elements of the two diagonal matrices. Toward this goal, our first step is to demonstrate that the leading-order term $\hat{\boldsymbol{\mathcal{S}}}^k_{D}$ of the boundary integral operator $\boldsymbol{\mathcal{S}}^{k}_{D}$ is invertible and that the two matrices are indeed diagonal.

\begin{lemm}
Let $D$ be a disk of radius $R$. For $k\neq 0$, the operator $\hat{\boldsymbol{\mathcal{S}}}^k_{D}:L^{2}(\partial D)^2\rightarrow H^{1}(\partial D)^2$ defined in \eqref{S:hat} is invertible.
\end{lemm}

\begin{proof}
Let us first demonstrate the injectivity of the operator $\hat{\boldsymbol{\mathcal{S}}}^k_{D}.$ Suppose  $\boldsymbol \varphi\in L^2(\partial D)^2$ satisfies
\begin{align}\label{SD0}
\hat{\boldsymbol{\mathcal{S}}}^k_{D}[\boldsymbol \varphi]=\boldsymbol{\mathcal{S}}_{D}[\boldsymbol \varphi]+\beta_k\int_{\partial D}\boldsymbol \varphi(\boldsymbol y)\mathrm{d}\sigma(\boldsymbol y)=\boldsymbol 0,
\end{align}
we only need to prove that $\boldsymbol \varphi=\boldsymbol 0$.

We consider the two cases identified previously:

\noindent
\textbf{Case 1:} $\tau_1 R \ln R = \dfrac{\tau_2}{2} R$.

Taking the inner product of both sides of \eqref{SD0} with $\boldsymbol f^{(i)} (i=1,2)$, we obtain 
\begin{align*}
&\left\langle\hat{\boldsymbol {\mathcal{S}}}^k_{D}[\boldsymbol \varphi], \boldsymbol f^{(1)}\right\rangle=\left\langle\boldsymbol{\mathcal{S}}_{D}[\boldsymbol \varphi], \boldsymbol f^{(1)}\right\rangle+\beta_k\left\langle\int_{\partial D}\boldsymbol \varphi(\boldsymbol y)\mathrm{d}\sigma(\boldsymbol y), \boldsymbol f^{(1)}\right\rangle=0,\\
&\left\langle\hat{\boldsymbol {\mathcal{S}}}^k_{D}[\boldsymbol \varphi], \boldsymbol f^{(2)}\right\rangle=\left\langle\boldsymbol{\mathcal{S}}_{D}[\boldsymbol \varphi], \boldsymbol f^{(2)}\right\rangle+\beta_k\left\langle\int_{\partial D}\boldsymbol \varphi(\boldsymbol y)\mathrm{d}\sigma(\boldsymbol y), \boldsymbol f^{(2)}\right\rangle=0.
\end{align*}
Given that the operator $\boldsymbol {\mathcal{S}}_{D}$ is self-adjoint and that $\boldsymbol {\mathcal{S}}_{D}[\boldsymbol f^{(1)}]=\boldsymbol {\mathcal{S}}_{D}[\boldsymbol f^{(2)}]=\boldsymbol 0$, it follows that
\[
\left\langle\int_{\partial D}\boldsymbol \varphi(\boldsymbol y)\mathrm{d}\sigma(\boldsymbol y), \boldsymbol f^{(1)}\right\rangle=0\quad \text{and}\quad
\left\langle\int_{\partial D}\boldsymbol \varphi(\boldsymbol y)\mathrm{d}\sigma(\boldsymbol y), \boldsymbol f^{(2)}\right\rangle=0,\]
which implies
\begin{align}\label{varphi:Case11}
\int_{\partial D}\boldsymbol \varphi(\boldsymbol y)\mathrm{d}\sigma(\boldsymbol y)=\boldsymbol 0.
\end{align}
Consequently, one has $\boldsymbol {\mathcal{S}}_{D}[\boldsymbol \varphi]= \boldsymbol 0$, and thus 
\begin{align}\label{varphi:Case12}
\boldsymbol \varphi=a_1\boldsymbol f^{(1)}+b_1\boldsymbol f^{(2)}
\end{align}
for some  constants $a_1$ and $b_1$ needing to be determined. Substituting \eqref{varphi:Case12} into \eqref{varphi:Case11}, yields 
\begin{align*}
\frac{a_1}{\sqrt{2\pi R}} \int_{\partial D} \begin{pmatrix} 1 \\ 0 \end{pmatrix} \mathrm{d}\sigma(\boldsymbol y) + \frac{b_1}{\sqrt{2\pi R}} \int_{\partial D} \begin{pmatrix} 0 \\ 1 \end{pmatrix} \mathrm{d}\sigma(\boldsymbol y) = \begin{pmatrix} 0 \\ 0 \end{pmatrix},
\end{align*}
which implies $a_1 = 0$ and $b_1 = 0$. Hence, $\boldsymbol{\varphi} =\boldsymbol 0$.

\noindent
\textbf{Case 2:} $\tau_1 R \ln R \neq \dfrac{\tau_2}{2} R$.

From \eqref{SD0}, we have 
\begin{align}\label{varphi:Case21}
\boldsymbol {\mathcal{S}}_{D}[\boldsymbol \varphi]=-\beta_k\int_{\partial D}\boldsymbol \varphi(\boldsymbol y)\mathrm{d}\sigma(\boldsymbol y).
\end{align}
By Lemma \ref{S:properties}, it follows  that
\begin{align}\label{varphi:Case22}
\boldsymbol \varphi=a_2\boldsymbol f^{(1)}+b_2\boldsymbol f^{(2)}.
\end{align}
Substituting \eqref{varphi:Case22} into \eqref{varphi:Case21} yields
\begin{align*}
\left\{
\begin{array}{l}
a_2\left(\tau_1 R\ln R-\frac{\tau_2}{2} R+\beta_k\sqrt{2\pi R}\right)= 0,\\
b_2\left(\tau_1 R\ln R-\frac{\tau_2}{2} R+\beta_k\sqrt{2\pi R}\right)= 0.
\end{array}
\right.
\end{align*}
Due to the fact that $\beta_k$ has a nonzero imaginary component, the coefficient $\tau_1 R\ln R-\frac{\tau_2}{2} R+\beta_k\sqrt{2\pi R}\neq 0$.
It follows that $a_2=0$ and $b_2=0$. Hence, $\boldsymbol \varphi=\boldsymbol 0.$

Having established injectivity in both cases, we now address invertibility.
Notice that $\boldsymbol{\mathcal{S}}_{D}$ is a Fredholm operator with index $0$ and $\beta_k\int_{\partial D}\boldsymbol \varphi(\boldsymbol y)\mathrm{d}\sigma(\boldsymbol y)$ is a finite-rank operator. Then, according to Proposition 1.4 in \cite{Lay}, we can obtain that  $\hat{\boldsymbol{\mathcal{S}}}^k_{D}$ is also a Fredholm operator of index $0$. Hence, based on the injectivity of $\hat{\boldsymbol{\mathcal{S}}}^k_{D}$, we can obtain that $\hat{\boldsymbol{\mathcal{S}}}^k_{D}$ is an invertible operator.
\end{proof}

Define the two $3\times3$ matrix
\begin{align*}
\boldsymbol Q=(Q_{ij})^3_{i,j=1} \quad\text{and}\quad \boldsymbol P=(P_{ij})^3_{i,j=1},
\end{align*}
where
\begin{align*}
Q_{ij}=\left\langle\boldsymbol f^{(i)},\boldsymbol{\mathcal{K}}^{(1)}_{D,1}[\boldsymbol f^{(j)}]\right\rangle \quad\text{and}\quad P_{ij}=\left\langle\boldsymbol f^{(i)},\boldsymbol{\mathcal{K}}^{(1)}_{D,2}[\boldsymbol f^{(j)}]\right\rangle.
\end{align*}

\begin{lemm}\label{QP}
Let $D$ be a disk of radius $R$. The matrix $\boldsymbol Q$ and $\boldsymbol P$ are diagonal,  with their elements given by
\begin{align}\label{q12}
Q_{11}=Q_{22}=-R^2\frac{\lambda^2+6\mu^2+5\lambda\mu}{4\mu(\lambda+2\mu)^2}, \quad Q_{33}=R^2\frac{2\lambda^2+7\mu^2+7\lambda\mu}{4\mu(\lambda+2\mu)^2}.
\end{align}
and
\begin{align}\label{p1}
P_{11}=P_{22}=&(\lambda+\mu)\pi R^2\left(2\sigma_1+3\sigma_2+\frac{(8\tau_1\tau_2-2\mu^{-2})(\ln R+1)-\tau_1\tau_2}{8\pi}\right)\nonumber\\
&+\mu\pi R^2\left(4\sigma_1+2\sigma_2+\left(\frac{\tau_1\tau_2}{\pi}-\frac{1}{2\pi\mu^2}\right)(\ln R+1)\right),
\end{align}
\begin{align}\label{p2}
P_{33}=&(\lambda+\mu)\pi R^2\left(2\sigma_1+3\sigma_2+\frac{\tau_1\tau_2(4\ln R+1)}{4\pi}-\frac{2\ln R+1}{8\pi\mu^2}\right)\nonumber\\
&+\mu\pi R^2\left(-4\sigma_1+2\sigma_2+\frac{\tau_1\tau_2}{2\pi}(1-\ln R)+\frac{\ln R+1}{4\pi\mu^2}\right),
\end{align}
where $\sigma_1,\sigma_2$ are defined in Lemma \ref{le:series} and $\tau_1,\tau_2$ are given by \eqref{parameters}.
\end{lemm}

\begin{proof}
According to the definition of $\boldsymbol{\mathcal{K}}^{(1),*}_{D,1}$, it holds that
$$\boldsymbol{\mathcal{K}}^{(1)}_{D,1}[\boldsymbol f^{(j)}]=\int_{\partial D}\partial_{\boldsymbol{\nu}_{\boldsymbol y}}\overline{\boldsymbol A(\boldsymbol x-\boldsymbol y)}\boldsymbol f^{(j)}(\boldsymbol y)\mathrm{d}\sigma(\boldsymbol y).$$
Consequently, the matrix elements $Q_{ij}$ are given by
\begin{align}\label{q}
Q_{ij}&=\langle\boldsymbol f^{(i)},\boldsymbol{\mathcal{K}}^{(1)}_{D,1}[\boldsymbol f^{(j)}]\rangle\nonumber\\
&=\int_{\partial D}\int_{\partial D}\boldsymbol f^{(i)}(\boldsymbol x)\cdot\partial_{\boldsymbol{\nu}_{\boldsymbol y}}\overline{\boldsymbol A(\boldsymbol x-\boldsymbol y)}\boldsymbol f^{(j)}(\boldsymbol y)\mathrm{d}\sigma(\boldsymbol y)\mathrm{d}\sigma(\boldsymbol x)\nonumber\\
&=\int_{\partial D}\int_D\mathcal{L}^{\lambda,\mu}_{\boldsymbol y}[\boldsymbol A(\boldsymbol x-\boldsymbol y)\boldsymbol f^{(j)}(\boldsymbol y)]\mathrm{d}\boldsymbol y\cdot\boldsymbol f^{(i)}(\boldsymbol x)\mathrm{d}\sigma(\boldsymbol x),
\end{align}
where we apply the integration by parts formula for the Lam\'{e} operator 
\[\int_D\mathcal{L}^{\lambda,\mu}_{\boldsymbol y}\boldsymbol u\cdot \overline{\boldsymbol v}\mathrm{d}\boldsymbol y+\int_D\lambda(\nabla_{\boldsymbol y}\cdot \boldsymbol u)(\overline{\nabla_{\boldsymbol y}\cdot \boldsymbol v})+\frac{\mu}{2}(\nabla_{\boldsymbol y}\boldsymbol u+\nabla_{\boldsymbol y}\boldsymbol u^\top):\overline{(\nabla_{\boldsymbol y}\boldsymbol v+\nabla_{\boldsymbol y}\boldsymbol v^\top)}\mathrm{d}\boldsymbol y=\int_{\partial D}\partial_{\boldsymbol{\nu}_{\boldsymbol y}}\boldsymbol u\cdot\overline{\boldsymbol v}\mathrm{d}\sigma(\boldsymbol y).\]
Here,  the Frobenius inner product for matrices is defined by $\boldsymbol A:\boldsymbol B=\mathrm{tr}(\boldsymbol{AB}^{\top})$ for square matrices $\boldsymbol A$ and $\boldsymbol B$,  and  $\mathcal{L}^{\lambda,\mu}_{\boldsymbol y}$ denotes the Lam\'{e} operator with respect to the variable $\boldsymbol y$.
Based on \eqref{q} and the orthonormality relation $\langle\boldsymbol f^{(i)},\boldsymbol f^{(j)}\rangle=\delta_{ij}$, we obtain the expressions in  \eqref{q12} and  conclude that
\[ Q_{ij}=0, \quad \text {for}~i\neq j.\]

Similarly, for the matrix $\boldsymbol{P}$, we have
\[P_{ij}=\int_{\partial D}\int_D\mathcal{L}^{\lambda,\mu}_{\boldsymbol y}[\boldsymbol B(\boldsymbol x-\boldsymbol y)\boldsymbol f^{(j)}(\boldsymbol y)]\mathrm{d}\boldsymbol y\cdot\boldsymbol f^{(i)}(\boldsymbol x)\mathrm{d}\sigma(\boldsymbol x).\]
Careful calculations yield the expressions \eqref{p1}  and \eqref{p2}.  Furthermore,  using rotational invariance of the disk, it follows that 
\[P_{ij}=0, \quad \text {for}~~i\neq j.\]
This completes the proof.
\end{proof}

Based on the layer potential theory \cite[p.28; Chapter 2.4]{Lay}, the solution to \eqref{Lame} can be represented as
\begin{align}\label{u}
\boldsymbol u(\boldsymbol x) =
\begin{cases}
\widetilde{\boldsymbol{\mathcal{S}}}_D^{\sqrt{\rho}\tau\omega}[\boldsymbol\phi](\boldsymbol x), & \boldsymbol x \in D,\\
\boldsymbol u^{\mathrm{in}} + \widetilde{\boldsymbol{\mathcal{S}}}_D^{\sqrt{\rho}\omega}[\boldsymbol \varphi](\boldsymbol x), & \boldsymbol x \in \mathbb{R}^2 \backslash \overline{D},
\end{cases}
\end{align}
for some densities $\boldsymbol\phi,\boldsymbol \varphi\in L^2(\partial D)^2$. With the help of the jump relations \eqref{Jump2}, $\boldsymbol\phi$ and $\boldsymbol \varphi$ satisfy the following boundary integral equations:
\begin{align}\label{AF:equation}
\boldsymbol{\mathcal{A}}(\omega, \delta)[\boldsymbol\Psi_{\boldsymbol F}] = \boldsymbol F,
\end{align}
where
\begin{align}\label{AF}
\boldsymbol{\mathcal{A}}(\omega, \delta) &=
\begin{pmatrix}
\boldsymbol{\mathcal{S}}_{D}^{\sqrt{\rho}\tau\omega} & -\boldsymbol{\mathcal{S}}_{D}^{\sqrt{\rho}\omega} \\
-\frac{1}{2}\boldsymbol{\mathcal{I}} + \boldsymbol{\mathcal{K}}_{D}^{\sqrt{\rho}\tau\omega,*} & -\delta\left(\frac{1}{2}\boldsymbol {\mathcal{I}} + \boldsymbol{\mathcal{K}}_{D}^{\sqrt{\rho}\omega,*}\right)
\end{pmatrix},\quad
\boldsymbol\Psi_{\boldsymbol F} =
\begin{pmatrix}
\boldsymbol\phi \\
\boldsymbol\varphi
\end{pmatrix},\quad
\boldsymbol F =
\begin{pmatrix}
\boldsymbol u^{\mathrm{in}} \\
\delta\partial_{\boldsymbol \nu}\boldsymbol u^{\mathrm{in}}
\end{pmatrix}.
\end{align}
Clearly, $\boldsymbol{\mathcal{A}}(\omega, \delta)$ is a bounded linear operator from $\mathcal{H}$ to $\mathcal{H}_1$.

\begin{defi}\label{D:reson}
The subwavelength resonant frequency is defined as the value of $\omega(\delta)\in\mathbb{C}$ for which  there exists a nontrivial solution $\boldsymbol{\Psi}_{\delta}$ to the homogeneous equation 
\begin{align*}
\boldsymbol{\mathcal{A}}(\omega, \delta)[\boldsymbol{\Psi}_{\delta}]=\boldsymbol 0,
\end{align*}
with the additional condition that  $\omega(\delta)\rightarrow 0$ as $\delta\rightarrow 0.$
\end{defi}

Moving forward, we now  determine the subwavelength resonant frequencies by means of asymptotic analysis.
From Lemma \ref{SK:asy}, we obtain the following asymptotic expression
\begin{align}\label{AB:asy}
\boldsymbol {\mathcal{A}}(\omega, \delta):=& \boldsymbol{\mathcal{A}}_0 + \boldsymbol{\mathcal{B}}(\omega, \delta)\nonumber\\
=& \boldsymbol{\mathcal{A}}_0 + \omega^2\ln\omega \boldsymbol{\mathcal{A}}_{1,1;0} + \omega^2 \boldsymbol{\mathcal{A}}_{1,2;0} + \delta \boldsymbol{\mathcal{A}}_{0,0;1} + \mathcal{O}(\delta\omega^2\ln\omega) + \mathcal{O}(\omega^4\ln\omega),
\end{align}
where
\begin{align*}
\boldsymbol{\mathcal{A}}_0 &= \begin{pmatrix}
\hat{\boldsymbol{\mathcal{S}}}_D^{\sqrt{\rho}\tau\omega} & -\hat{\boldsymbol{\mathcal{S}}}_D^{\sqrt{\rho}\omega} \\
-\frac{1}{2}\boldsymbol{\mathcal{I}} + \boldsymbol{\mathcal{K}}_D^* &  \boldsymbol 0
\end{pmatrix},\qquad \qquad \qquad \qquad \qquad \qquad \qquad
\boldsymbol{\mathcal{A}}_{1,1;0} = \begin{pmatrix}
\rho \tau^2 \boldsymbol{\mathcal{S}}_{D,1}^{(1)} & -\rho \boldsymbol{\mathcal{S}}_{D,1}^{(1)} \\
\rho \tau^2 \boldsymbol{\mathcal{K}}_{D,1}^{(1),*} &  \boldsymbol 0
\end{pmatrix}, \\
\boldsymbol{\mathcal{A}}_{1,2;0} &= \begin{pmatrix}
\rho \tau^2 \ln(\sqrt{\rho}\tau) \boldsymbol{\mathcal{S}}_{D,1}^{(1)} + \rho \tau^2 \boldsymbol{\mathcal{S}}_{D,1}^{(2)} & -\rho \ln \sqrt{\rho} \boldsymbol{\mathcal{S}}_{D,1}^{(1)} - \rho \boldsymbol{\mathcal{S}}_{D,1}^{(2)} \\
\rho \tau^2 \ln(\sqrt{\rho}\tau) \boldsymbol{\mathcal{K}}_{D,1}^{(1),*} + \rho \tau^2 \boldsymbol{\mathcal{K}}_{D,1}^{(2),*} &  \boldsymbol 0
\end{pmatrix},\quad
\boldsymbol{\mathcal{A}}_{0,0;1} = \begin{pmatrix}
 \boldsymbol 0 &  \boldsymbol 0 \\
 \boldsymbol 0 & -\left(\frac{1}{2}\boldsymbol{\mathcal{I}} + \boldsymbol{\mathcal{K}}_D^*\right)
\end{pmatrix}.
\end{align*}

\begin{lemm}\label{A0:ker}
It holds that
\begin{itemize}
\item[(i)] $\operatorname {Ker}{\boldsymbol{\mathcal{A}}_0}= \operatorname{span}\{\boldsymbol \Psi_i\}$, with  
    \[\boldsymbol\Psi_i=b_i\begin{pmatrix} \boldsymbol f^{(i)}\\a_i\boldsymbol f^{(i)}\end{pmatrix},\]
where  the coefficients are given by 
\begin{align}\label{a}
&a_1=a_2=\left\{ \begin{aligned}
&\frac{\beta_{\sqrt{\rho}\tau\omega}}{\beta_{\sqrt{\rho}\omega}}&& \text{in Case 1 },\\
&\frac{\tau_1 R\ln R-\frac{\tau_2}{2} R+\sqrt{2\pi R}\beta_{\sqrt{\rho}\tau\omega}}{\tau_1 R\ln R-\frac{\tau_2}{2} R+\sqrt{2\pi R}\beta_{\sqrt{\rho}\omega}}&&\text{in Case 2 },
\end{aligned}\right.\\
&a_3=1,\nonumber
\end{align}
and where the constants $b_i$ are chosen such that $\langle \boldsymbol{\Psi}_i, \boldsymbol{\Psi}_j \rangle_{\mathcal{H}} = \delta_{ij}$.

\item[(ii)]
$\operatorname {Ker}{\boldsymbol{\mathcal{A}}^*_0}=\operatorname {span}\{\boldsymbol \Phi_i\}$, with 
$$\boldsymbol\Phi_i=\begin{pmatrix}  \boldsymbol 0\\\boldsymbol f^{(i)}\end{pmatrix}.$$
\end{itemize}
\end{lemm}

\begin{proof}
Assuming that 
\begin{align*}
\operatorname{ker}{\boldsymbol{\mathcal{A}}_0}=\text{span}\left\{\begin{pmatrix}\boldsymbol{\phi}_i\\
\boldsymbol{\varphi}_i
\end{pmatrix}\right\},
\end{align*}
it holds that
\begin{align}\label{ker1}
\left\{
\begin{array}{l}
\boldsymbol{\mathcal{S}}_D[\boldsymbol{\phi}_i] + \beta_{\sqrt{\rho}\tau\omega} \int_{\partial D} \boldsymbol{\phi_i}(\boldsymbol y)\mathrm{d}\sigma(\boldsymbol y) - \boldsymbol{\mathcal{S}}_D[\boldsymbol{\varphi}_i] - \beta_{\sqrt{\rho} \omega} \int_{\partial D} \boldsymbol{\varphi}_i(\boldsymbol y)\mathrm{d}\sigma(\boldsymbol y) = 0,\\
\left(-\frac{1}{2}\boldsymbol{\mathcal{I}} + \boldsymbol{\mathcal{K}}_D^* \right) [\boldsymbol{\phi}_i] = 0.
\end{array}
\right.
\end{align}
From the second equality in \eqref{ker1} and Lemma \ref{SK:ker},  it follows that $\boldsymbol \phi_i$ is a multiple of $\boldsymbol f^{(i)},(i=1,2,3)$. Without loss of generality, we set $\boldsymbol \phi_i=\boldsymbol f^{(i)}$  and proceed to  determine the corresponding $\boldsymbol\varphi_i$.

In Case 1,
as $\boldsymbol\phi_i=\boldsymbol f^{(i)}(i=1,2)$, it follows from \eqref{ker1} that
$$\beta_{\sqrt{\rho}\tau\omega}\int_{\partial D}\boldsymbol f^{(i)}\mathrm{d}\sigma(\boldsymbol y)=\boldsymbol{\mathcal{S}}[\boldsymbol\varphi_i]+\beta_{\sqrt{\rho}\omega}\int_{\partial D}\boldsymbol \varphi_i\mathrm{d}\sigma(\boldsymbol y).$$
Then, Lemma \ref{S:properties} indicates that $\boldsymbol\varphi_i=c\boldsymbol f^{(i)}$ for some constant $c$, which  satisfies 
\[\sqrt{2\pi R}\beta_{\sqrt{\rho}\tau\omega}=c\beta_{\sqrt{\rho}\omega}\sqrt{2\pi R}.\] 
 One obtains 
 \begin{align*}
 \boldsymbol\varphi_i=\frac{\beta_{\sqrt{\rho}\tau\omega}}{\beta_{\sqrt{\rho}\omega}}\boldsymbol f^{(i)}, \quad\text{for }i=1,2.
\end{align*}
As $\boldsymbol \phi_{3}=\boldsymbol f^{(3)},$ we can get from \eqref{ker1} that
$$\boldsymbol{\mathcal{S}}[\boldsymbol f^{(3)}]=\boldsymbol{\mathcal{S}}[\boldsymbol\varphi_3]+\beta_{\sqrt{\rho}\omega}\int_{\partial D}\boldsymbol{\varphi}_3\mathrm{d}\sigma(\boldsymbol y),$$
where we use the fact for a disk $D$ that
\begin{align}\label{f30}
\int_{\partial D}\boldsymbol f^{(3)}(\boldsymbol y)\mathrm{d}\sigma(\boldsymbol y)= \boldsymbol 0,
\end{align}
 Then, it holds that
$$\boldsymbol\varphi_3=\boldsymbol f^{(3)}.$$

In Case 2,
as $\boldsymbol\phi_i=\boldsymbol f^{(i)}(i=1,2)$, one has
$$\boldsymbol{\mathcal{S}}[\boldsymbol f^{(i)}-\boldsymbol{\varphi}_{i}]=\beta_{\sqrt{\rho}\omega}\int_{\partial D}\boldsymbol \varphi_i\mathrm{d}\sigma(\boldsymbol y)-\beta_{\sqrt{\rho}\tau\omega}\int_{\partial D}\boldsymbol{f}^{(i)}\mathrm{d}\sigma(\boldsymbol y).$$
It follows from Lemma \ref{S:properties} that $\boldsymbol\varphi_i=c\boldsymbol f^{(i)}$. Moreover, the constant $c$ satisfies
\begin{align*}
(1-c)\left(\tau_1 R\ln R-\frac{\tau_2}{2} R\right)
=c\beta_{\sqrt{\rho}\omega}\sqrt{2\pi R}-\beta_{\sqrt{\rho}\tau\omega}\sqrt{2\pi R}.
\end{align*}
Then, we get
\begin{align*}
\boldsymbol\varphi_i=\frac{\tau_1 R\ln R-\frac{\tau_2}{2} R+\sqrt{2\pi R}\beta_{\sqrt{\rho}\tau\omega}}{\tau_1 R\ln R-\frac{\tau_2}{2} R+\sqrt{2\pi R}\beta_{\sqrt{\rho}\omega}}\boldsymbol f^{(i)},\quad \text{for }i=1,2.
\end{align*}
As $\boldsymbol \phi_{3}=\boldsymbol f^{(3)},$ following the same analysis as in Case 1, we conclude that
$$\boldsymbol\varphi_3=\boldsymbol f^{(3)}.$$
We have completed the proof for (i).

For part (ii), note that the adjoint operator is given by
\begin{align*}
\boldsymbol{\mathcal{A}}^*_0 &= \begin{pmatrix}
\hat{\boldsymbol{\mathcal{S}}}_D^{\sqrt{\rho}\tau\omega,*} &-\frac{1}{2}\boldsymbol{\mathcal{I}} + \boldsymbol{\mathcal{K}}_D \\
 -\hat{\boldsymbol{\mathcal{S}}}_D^{\sqrt{\rho}\omega,*} & \boldsymbol 0
\end{pmatrix}.
\end{align*}
Given the invertibility of $\hat{\boldsymbol{\mathcal{S}}}^k_D$, it follows that the first component of $\boldsymbol \Phi_i$ must be zero. Lemma \ref{SK:ker} indicates the second component of $\boldsymbol \Phi_i$ is $\boldsymbol f^{(i)}$.

The proof is completed.
\end{proof}
We define a new operator $\tilde{\boldsymbol{\mathcal{A}}_0}: \mathcal{H} \to \mathcal{H}_1$ by
\[\tilde{\boldsymbol{\mathcal{A}}_0}:=\boldsymbol{\mathcal{A}}_0+\boldsymbol{\mathcal{M}},\]
 where $\boldsymbol{\mathcal{M}} : \mathcal{H} \to \mathcal{H}_1 $ given by
 \[\boldsymbol{\mathcal{M}}[\boldsymbol \psi]=\sum^3_{i=1}\langle\boldsymbol\psi,\boldsymbol\Psi_i\rangle_{\mathcal{H}}\boldsymbol\Phi_i.\]
  By construction, this operator satisfies $\tilde{\boldsymbol{\mathcal{A}}_0}[\boldsymbol \Psi_i]=\boldsymbol \Phi_i$. Notice that $\boldsymbol{\mathcal{M}}^*(\boldsymbol\theta)=\sum^3_{i=1}\left\langle\boldsymbol\theta, \boldsymbol \Phi_i\right\rangle_{\mathcal{H}}\boldsymbol \Psi_i$. It holds that
$\tilde{\boldsymbol{\mathcal{A}}_0}^*[\boldsymbol \Phi_i]=\boldsymbol \Psi_i$. Given the construction of $\tilde{\boldsymbol{\mathcal{A}}_0}$ and the invertibility of $\hat{\boldsymbol{\mathcal{S}}}^k_D$, it follows that $\tilde{\boldsymbol{\mathcal{A}}_0}$ is bijective, which results in $\tilde{\boldsymbol{\mathcal{A}}_0}^*$ being a bijection.

\begin{theo}\label{theo:fre}
Let $D$ be a disk of radius $R$. For $\delta,\epsilon\in\mathbb{R}^+\rightarrow0$, there exist three subwavelength resonant frequencies (counted with their multiplicities), whose leading-order terms denoted by $\hat{\omega}_i (i=1,2,3)$  can be determined by the equations as follows:
\begin{align}\label{resonant:fre}
\rho\hat{\omega}^2_i\ln\hat{\omega}_i Q_{ii}+\rho\hat{\omega}_i^2\left(\ln\left(\sqrt{\rho}\tau\right)Q_{ii}+P_{ii}\right)-\epsilon a_i=0,\quad \text{for }i=1,2,3,
\end{align}
where the constants  $a_i$ are defined in Lemma \ref{A0:ker}, and $Q_{ii},P_{ii}$ are given in Lemma \ref{QP}.
\end{theo}

\begin{proof}
According to Definition \ref{D:reson},  the subwavelength resonant frequency $\omega^* \ll 1$ is characterized by the existence of a nontrivial solution $\boldsymbol{\Psi}_{\delta}$ to the equation
\begin{align}\label{A0}
\boldsymbol{\mathcal{A}}(\omega^*,\delta)[\boldsymbol{\Psi}_{\delta}]=\boldsymbol 0.
\end{align}
Set $\boldsymbol\Psi=\sum^3_{j=1}c_j\boldsymbol\Psi_j$ with the constants $c_j$. Due to $\operatorname{Ker}{\boldsymbol{\mathcal{A}}_0}= \text{span}\{\boldsymbol \Psi_j\}$,
$\boldsymbol{\Psi}_{\delta}$ is a slight perturbation of $\boldsymbol{\Psi}$. Hence, we can decompose $\boldsymbol{\Psi}_{\delta}$ as
\begin{align*}
\boldsymbol{\Psi}_{\delta}=\boldsymbol{\Psi}+\boldsymbol{\Psi}^\bot \quad\text{with} \quad\langle\boldsymbol{\Psi},\boldsymbol{\Psi}^\bot\rangle_{\mathcal{H}}=0.
\end{align*}
It follows from \eqref{A0} that
\begin{align}\label{A1}
(\tilde{\boldsymbol{\mathcal{A}}_0}-\boldsymbol{\mathcal{M}}+\boldsymbol{\mathcal{B}}(\omega,\delta))[\boldsymbol{\Psi}+\boldsymbol{\Psi}^\bot]= \boldsymbol 0.
\end{align}
With $\tilde{\boldsymbol{\mathcal{A}}_0}$ being invertible, for sufficiently small $\omega$ and $\delta$, the perturbed operator $\tilde{\boldsymbol{\mathcal{A}}_0}+\boldsymbol{\mathcal{B}}(\omega,\delta)$ remains invertible. Applying $(\tilde{\boldsymbol{\mathcal{A}}_0}+\boldsymbol{\mathcal{B}}(\omega,\delta))^{-1}$ to both sides of $\eqref{A1}$ gives
\begin{align*}
\boldsymbol{\Psi}^\bot=(\tilde{\boldsymbol{\mathcal{A}}_0}+\boldsymbol{\mathcal{B}})^{-1}\boldsymbol{\mathcal{M}}[\boldsymbol \Psi]-\boldsymbol \Psi.
\end{align*}
Then, Neumann series yields
\begin{align*}
\langle\boldsymbol{\Psi}^\bot,\boldsymbol \Psi\rangle_{\mathcal{H}}=&\left\langle-\tilde{\boldsymbol{\mathcal{A}}_0}^{-1}\left(\omega^2\ln\omega\boldsymbol{\mathcal{A}}_{1,1;0}+\omega^2\boldsymbol{\mathcal{A}}_{1,2;0}+\delta\boldsymbol{\mathcal{A}}_{0,0;1}+\mathcal{O}(\delta\omega^2\ln\omega+\omega^4\ln\omega+\omega^4)\right)\boldsymbol \Psi,\boldsymbol \Psi\right\rangle_{\mathcal{H}}\\
=&-\left\langle(\omega^2\ln\omega\boldsymbol{\mathcal{A}}_{1,1;0}+\omega^2\boldsymbol{\mathcal{A}}_{1,2;0}+\delta\boldsymbol{\mathcal{A}}_{0,0;1})[\boldsymbol\Psi],\sum^3_{j=1}c_j\boldsymbol\Phi_j\right\rangle_{\mathcal{H}_1}\\
&+\mathcal{O}(\delta\omega^2\ln\omega+\omega^4\ln\omega+\omega^4),
\end{align*}
where we use $\tilde{\boldsymbol{\mathcal{A}}_0}^*[\boldsymbol \Phi_j]=\boldsymbol \Psi_j$.

Next, we calculate $\langle\boldsymbol{\mathcal{A}}_{1,1;0}[\boldsymbol\Psi],\sum^3_{j=1}c_j\boldsymbol\Phi_j\rangle_{\mathcal{H}_1}$, $\langle\boldsymbol{\mathcal{A}}_{1,2;0}[\boldsymbol\Psi],\sum^3_{j=1}c_j\boldsymbol\Phi_j\rangle_{\mathcal{H}_1}$ and $\langle\boldsymbol{\mathcal{A}}_{0,0;1}[\boldsymbol\Psi],\sum^3_{j=1}c_j\boldsymbol\Phi_j\rangle_{\mathcal{H}_1}$.
It holds that
\begin{align*}
\left\langle\boldsymbol{\mathcal{A}}_{1,1;0}[\boldsymbol\Psi],\sum^3_{j=1}c_j\boldsymbol\Phi_j\right\rangle_{\mathcal{H}_1}&=\left\langle\sum^3_{i=1}c_i\boldsymbol\Psi_i,\sum^3_{j=1}c_j\boldsymbol{\mathcal{A}}^*_{1,1;0}\boldsymbol\Phi_j\right\rangle_{\mathcal{H}}\\
&=\left\langle\sum^3_{i=1}c_ib_i\begin{pmatrix} \boldsymbol f^{(i)}\\a_i\boldsymbol f^{(i)}\end{pmatrix},\sum^3_{j=1}c_j\begin{pmatrix}\rho\tau^2\boldsymbol{\mathcal{K}}^{(1)}_{D,1}[\boldsymbol f^{(j)}]\\  \boldsymbol 0 \end{pmatrix}\right\rangle_{\mathcal{H}}\\
&=\sum^3_{i,j=1}c_ib_i\rho\tau^2c_j\left\langle\boldsymbol f^{(i)},\boldsymbol{\mathcal{K}}^{(1)}_{D,1}[\boldsymbol f^{(j)}]\right\rangle,\\
&=\sum^3_{i,j=1}c_ib_i\rho\tau^2c_jQ_{ij}.
\end{align*}
Similarly, we have
\begin{align*}
\left\langle\boldsymbol{\mathcal{A}}_{1,2;0}[\boldsymbol\Psi],\sum^3_{j=1}c_j\boldsymbol\Phi_j\right\rangle_{\mathcal{H}_1}=&\sum^3_{i,j=1}c_ic_jb_i\rho\tau^2\ln(\sqrt{\rho}\tau)\left\langle\boldsymbol f^{(i)},\boldsymbol{\mathcal{K}}^{(1)}_{D,1}[\boldsymbol f^{(j)}]\right\rangle\\
&+\sum^3_{i,j=1}c_ic_jb_i\rho\tau^2\left\langle\boldsymbol f^{(i)},\boldsymbol{\mathcal{K}}^{(1)}_{D,2}[\boldsymbol f^{(j)}]\right\rangle\\
=&\sum^3_{i,j=1}c_ib_ic_j\rho\tau^2\ln(\sqrt{\rho}\tau)Q_{ij}+\sum^3_{i,j=1}c_ib_ic_j\rho\tau^2P_{ij},
\end{align*}
and
\begin{align*}
\left\langle\boldsymbol{\mathcal{A}}_{0,0;1}[\boldsymbol\Psi],\sum^3_{j=1}c_j\boldsymbol\Phi_j\right\rangle_{\mathcal{H}_1}&=\left\langle\sum^3_{i=1}c_ib_i\begin{pmatrix} \boldsymbol 0\\-(\frac{1}{2}\boldsymbol{\mathcal{I}}+\boldsymbol{\mathcal{K}}^*_D)[a_i\boldsymbol f^{(i)}]\end{pmatrix},\sum^3_{j=1}c_j\begin{pmatrix} \boldsymbol 0 \\ \boldsymbol f^{(j)}\end{pmatrix}\right\rangle_{\mathcal{H}_1}\\
&=-\sum^3_{i,j=1}c_ib_ia_ic_j\rho\tau^2\left\langle\boldsymbol f^{(i)},\boldsymbol f^{(j)}\right\rangle\\
&=-\sum^3_{i,j=1}c_ib_ia_ic_j\rho\tau^2\delta_{ij}\\
&=-\sum^3_{i=1}c^2_ia_ib_i\rho\tau^2.
\end{align*}
Then, it follows that
\begin{align*}
0&=\langle\boldsymbol{\Psi}^\bot,\boldsymbol \Psi\rangle_{\mathcal{H}}\nonumber\\
&=-\boldsymbol c^\top\left(\rho\tau^2\omega^2\ln\omega\boldsymbol Q\boldsymbol\Pi+\rho\tau^2\omega^2\ln(\sqrt{\rho}\tau)\boldsymbol Q\boldsymbol\Pi+\rho\tau^2\omega^2\boldsymbol P\boldsymbol\Pi-\delta\boldsymbol N\right)\boldsymbol c+\mathcal{O}(\delta\omega^2\ln\omega+\omega^4\ln\omega+\omega^4),
\end{align*}
where $\boldsymbol{c}^{\top}=(c_1,c_2,c_3)$, and $\boldsymbol\Pi$ and $\boldsymbol N$ are diagonal matrix with entries $\Pi_{ii}=b_i$ and $N_{ii}=a_ib_i$, respectively.
Hence, Lemma \ref{QP} indicates that subwavelength resonant frequencies, denoted by $\omega_i (i=1,2,3)$, are determined by
$$\rho\tau^2\omega_i^2\ln \omega_i b_i Q_{ii}+\rho\tau^2\omega^2\ln(\sqrt{\rho}\tau)b_iQ_{ii}+\rho\tau^2\omega_i^2b_iP_{ii}-\delta a_ib_i+\mathcal{O}(\delta\omega^2\ln\omega+\omega^4\ln\omega+\omega^4)=0.$$
We can obtain \eqref{resonant:fre} from \eqref{contrast}.
\end{proof}


\section{The scattered fields}\label{sec:4}

This section aims to derive the asymptotic expansion of the scattered field within the resonator $D$ and characterize the far-field patterns, which include both longitudinal and transverse far-field patterns. Furthermore,  we analyze the behaviors of these scattered fields as the incident frequency $\omega$ lies in different regimes. A key finding is that when the incident frequency $\omega$ approaches a subwavelength resonant frequency identified in Theorem \ref{theo:fre}, the scattered field within $D$ exhibits a pronounced enhancement, scaling with $|\ln \omega|$.

\subsection{Solution to the internal problem}
Without loss of generality, we consider a time-harmonic compressed plane wave incident field of the form
\begin{align*}
\boldsymbol{u}^{\mathrm{in}}(\boldsymbol x)=\boldsymbol{d}e^{\mathrm{i}k_p\boldsymbol x\cdot \boldsymbol{d}},
\end{align*}
where $\boldsymbol{d} = (d_1, d_2)$ is a unit vector specifying the direction of propagation,  and the compressional wavenumber $k_p=\frac{\omega}{c_p}$. By adhering to the same analytical framework, we can arrive at similar conclusions for other incident waves. In order to obtain the asymptotic expansion for the field inside the resonator, we first  establish the following two lemmas.

\begin{lemm}\label{A0d:inver}
For $i=1,2$, we have
\begin{align*}
\tilde{\boldsymbol{\mathcal{A}}_0}^{-1}\begin{pmatrix} \sqrt{2\pi R}\boldsymbol f^{(i)}\\  \boldsymbol 0 \end{pmatrix}=\begin{pmatrix} \eta\boldsymbol f^{(i)} \\\widetilde{\eta} \boldsymbol f^{(i)}\end{pmatrix},
\end{align*}
where
\begin{align}\label{Case1}
\left\{ \begin{aligned}
&\eta=\frac{\beta_{\sqrt{\rho}\tau\omega}}{\beta^2_{\sqrt{\rho}\tau\omega}+\beta^2_{\sqrt{\rho}\omega}},\\
&\tilde{\eta}=\frac{-\beta_{\sqrt{\rho}\omega}}{\beta^2_{\sqrt{\rho}\tau\omega}+\beta^2_{\sqrt{\rho}\omega}},
\end{aligned}\right.\qquad \text{in Case 1}
\end{align}
or 
\begin{align}\label{Case2}
\left\{ \begin{aligned}
&\eta=\frac{\sqrt{2\pi R}a_1}{(a_1+1)\left(\tau_1 R\ln R-\frac{\tau_2}{2} R\right)+\sqrt{2\pi R}(\beta_{\sqrt{\rho}\tau\omega}a_1+\beta_{\sqrt{\rho}\omega})},\\
&\tilde{\eta}=\frac{-\sqrt{2\pi R}}{(a_1+1)\left(\tau_1 R\ln R-\frac{\tau_2}{2} R\right)+\sqrt{2\pi R}(\beta_{\sqrt{\rho}\tau\omega}a_1+\beta_{\sqrt{\rho}\omega})},
\end{aligned}\right. \qquad\text{in Case 2 }
\end{align}
with $a_1$ given by \eqref{a}.
\end{lemm}

\begin{proof}
Assume that
\begin{align*}
\tilde{\boldsymbol{\mathcal{A}}_0}\begin{pmatrix}\boldsymbol{\psi}_D\\ \boldsymbol{\psi} \end{pmatrix}=\begin{pmatrix} \sqrt{2\pi R}\boldsymbol f^{(i)} \\
\boldsymbol 0\end{pmatrix}.
\end{align*}
According to the definition of $\tilde{\boldsymbol{\mathcal{A}}_0}$, this is equivalent to the equation
\begin{align}\label{A0:inver}
\begin{pmatrix}
\hat{\boldsymbol{\mathcal{S}}}_D^{\sqrt{\rho}\tau\omega} & -\hat{\boldsymbol{\mathcal{S}}}_D^{\sqrt{\rho}\omega} \\
-\frac{1}{2}\boldsymbol{\mathcal{I}} + \boldsymbol{\mathcal{K}}_D^* &  \boldsymbol 0
\end{pmatrix}
\begin{pmatrix}
\boldsymbol{\psi}_D \\
\boldsymbol{\psi}
\end{pmatrix}
+\sum^3_{j=1}
\left\langle\begin{pmatrix}
\boldsymbol{\psi}_D \\
\boldsymbol{\psi}
\end{pmatrix},
b_j
\begin{pmatrix}
\boldsymbol f^{(j)} \\
a_j\boldsymbol f^{(j)}
\end{pmatrix}\right\rangle_{\mathcal{H}}
\begin{pmatrix}
 \boldsymbol 0 \\
\boldsymbol f^{(j)}
\end{pmatrix}
=
\begin{pmatrix}
\sqrt{2\pi R}\boldsymbol f^{(i)} \\
  \boldsymbol 0
\end{pmatrix}.
\end{align}
Based on Lemma \ref{S:properties} and (iii) in Lemma \ref{SK:ker}, we can assume that $\boldsymbol \psi_D=\eta_i\boldsymbol f^{(i)}$ and $\boldsymbol \psi=\tilde{\eta}_i\boldsymbol f^{(i)}$. Then, we need to determine the coefficients  $\eta_i$ and $\tilde{\eta}_i$. We now proceed  by two cases.

In Case 1 (i.e., $\tau_1 R\ln R=\frac{\tau_2}{2} R$), it follows that  $\boldsymbol{\mathcal{S}}_D[\boldsymbol f^{(1)}]=\boldsymbol{\mathcal{S}}_D[\boldsymbol f^{(2)}]= \boldsymbol 0.$
By \eqref{A0:inver}, we can obtain the system 
\begin{align*}
\begin{aligned}
\begin{cases}
\eta_i\beta_{\sqrt{\rho}\tau\omega} \sqrt{2\pi R}- \tilde{\eta}_i\beta_{\sqrt{\rho} \omega} \sqrt{2\pi R}= \sqrt{2\pi R},\\
\eta_i b_i+\tilde{\eta}_ib_ia_i=0,
\end{cases}
\end{aligned}
\end{align*}
where we utilize use (iii) in Lemma \ref{SK:ker} for the second equality.
Solving the above equations directly yields the result stated in  \eqref{Case1}.

In Case 2 (i.e., $\tau_1 R\ln R\neq\frac{\tau_2}{2} R.$), one has
$$\boldsymbol{\mathcal{S}}[\boldsymbol f^{(i)}]=\left(\tau_1 R\ln R-\frac{\tau_2}{2} R\right)\boldsymbol f^{(i)},\quad\text{for }i=1,2.$$
Substituting the above equality into \eqref{A0:inver} and using the orthonormality relation $\langle \boldsymbol{f}^{(i)}, \boldsymbol{f}^{(j)} \rangle = \delta_{ij}$ leads to the system
\begin{align*}
\left\{
\begin{array}{l}
(\eta_i-\tilde{\eta}_i)\left(\tau_1 R\ln R-\frac{\tau_2}{2} R\right)+\sqrt{2\pi R}(\eta_i\beta_{\sqrt{\rho}\tau\omega}-\tilde{\eta}_i\beta_{\sqrt{\rho}\omega})= \sqrt{2\pi R},\\
\eta_i b_i+\tilde{\eta}_ib_ia_i= 0.
\end{array}
\right.
\end{align*}
Solving this system gives the result stated in \eqref{Case2}.

The proof is complete.
\end{proof}

Next, we analyze the asymptotic behavior of the inner product
 $\langle(\tilde{\boldsymbol{\mathcal{A}}_0}+\boldsymbol {\mathcal{B}})^{-1}[\boldsymbol F],\boldsymbol{\Psi}_i\rangle_{\mathcal{H}}$ associated with the incident field  $\boldsymbol{u}^{\mathrm{in}}$. Here,  $\boldsymbol F$ and  $\boldsymbol{\Psi}_i$
  are defined by \eqref{AF} and Lemma \ref{A0:ker}, respectively.

\begin{lemm}
The following asymptotic expansion holds:
\begin{align}\label{dj:equation}
\left\langle\left(\tilde{\boldsymbol{\mathcal{A}}_0} + \boldsymbol{\mathcal{B}}\right)^{-1}[\boldsymbol{F}], \boldsymbol{\Psi}_i \right\rangle_{\mathcal{H}} =
\begin{cases}
-\omega^2 \ln \omega \, \rho \tau^2 \eta d_i Q_{ii} - \omega^2 \rho \tau^2 \eta \left( \ln(\sqrt{\rho} \tau) d_i Q_{ii} + d_i P_{ii} \right) \\
\quad +\, \delta \tilde{\eta} d_i + \mathcal{O}\left( \omega^3 \ln \omega + \omega^3 + \omega \delta \right), & i = 1, 2, \\[1.5ex]
\mathcal{O}\left( \omega^3 \ln \omega + \omega^3 + \omega \delta \right), & i = 3,
\end{cases}
\end{align}
where $\eta$ and $\tilde{\eta}$ are given in Lemma \ref{A0d:inver}, and $Q_{ii}$, $P_{ii}$ are given in Lemma \ref{QP}.
\end{lemm}

\begin{proof}

We decompose the source term $\boldsymbol{F}$ (defined in \eqref{AF}) as $\boldsymbol{F} = \boldsymbol{F}_1 + \boldsymbol{F}_2$, where
\[
\boldsymbol{F}_1 = \begin{pmatrix} \boldsymbol{u}^{\mathrm{in}} \\ \boldsymbol 0 \end{pmatrix} \quad \text{and} \quad \boldsymbol{F}_2 = \begin{pmatrix}  \boldsymbol 0 \\ \delta \partial_{\boldsymbol{\nu}} \boldsymbol{u}^{\mathrm{in}} \end{pmatrix}.
\]
From the properties of the incident field, we have the asymptotic expansions
\[
\boldsymbol{F}_1 = \begin{pmatrix} \boldsymbol{d} \\ \boldsymbol 0 \end{pmatrix} + \mathcal{O}(\omega) \quad \text{and} \quad \boldsymbol{F}_2 = \mathcal{O}(\omega\delta).
\]
Starting from the target inner product, we proceed term by term
\begin{align*}
&\left\langle \left(\tilde{\boldsymbol{\mathcal{A}}_0} + \boldsymbol{\mathcal{B}} \right)^{-1} [\boldsymbol{F}], \boldsymbol{\Psi}_i \right\rangle_{\mathcal{H}} \\
= &\left\langle \left( \tilde{\boldsymbol{\mathcal{A}}_0} + \boldsymbol{\mathcal{B}} \right)^{-1} [\boldsymbol{F}_1 + \boldsymbol{F}_2], \boldsymbol{\Psi}_i \right\rangle_{\mathcal{H}} \\
=& \left\langle \left( \tilde{\boldsymbol{\mathcal{A}}_0} + \boldsymbol{\mathcal{B}} \right)^{-1} [\boldsymbol{F}_1], \boldsymbol{\Psi}_i \right\rangle_{\mathcal{H}} + \mathcal{O}(\omega \delta) \quad \text{(since $\boldsymbol{F}_2 = \mathcal{O}(\omega\delta)$)}.
\end{align*}
Using the Neumann series expansion for the resolvent operator
\begin{align*}
&\left(\tilde{\boldsymbol{\mathcal{A}}_0} + \boldsymbol{\mathcal{B}} \right)^{-1} 
= \left(\boldsymbol{\mathcal{I}} + \tilde{\boldsymbol{\mathcal{A}}_0}^{-1} \boldsymbol{\mathcal{B}} \right)^{-1} \tilde{\boldsymbol{\mathcal{A}}_0}^{-1} \\
= &\left( \boldsymbol{\mathcal{I}} - \tilde{\boldsymbol{\mathcal{A}}_0}^{-1} \boldsymbol{\mathcal{B}} + \mathcal{O}(\|\tilde{\boldsymbol{\mathcal{A}}_0}^{-1} \boldsymbol{\mathcal{B}}\|^2) \right) \tilde{\boldsymbol{\mathcal{A}}_0}^{-1} \\
= &\tilde{\boldsymbol{\mathcal{A}_0}}^{-1} -\tilde{\boldsymbol{\mathcal{A}}_0}^{-1} \boldsymbol{\mathcal{B}} \tilde{\boldsymbol{\mathcal{A}}_0}^{-1} + \mathcal{O}(\delta\omega^2\ln\omega + \omega^4\ln\omega + \omega^4).
\end{align*}
Substituting this expansion yields
\begin{align*}
&\left\langle \left(\tilde{\boldsymbol{\mathcal{A}}_0} + \boldsymbol{\mathcal{B}} \right)^{-1} [\boldsymbol{F}_1], \boldsymbol{\Psi}_i \right\rangle_{\mathcal{H}} \\
=& \left\langle\tilde{\boldsymbol{\mathcal{A}}_0}^{-1} [\boldsymbol{F}_1], \boldsymbol{\Psi}_i \right\rangle_{\mathcal{H}} 
   - \left\langle \tilde{\boldsymbol{\mathcal{A}}_0}^{-1} \boldsymbol{\mathcal{B}} \tilde{\boldsymbol{\mathcal{A}}_0}^{-1} [\boldsymbol{F}_1], \boldsymbol{\Psi}_i \right\rangle_{\mathcal{H}} 
   + \mathcal{O}( \delta\omega^2\ln\omega+ \omega^4\ln\omega + \omega^4) \\
=& \left\langle \tilde{\boldsymbol{\mathcal{A}}_0}^{-1} [\boldsymbol{F}_1], \boldsymbol{\Psi}_i \right\rangle_{\mathcal{H}} 
   - \left\langle \boldsymbol{\mathcal{B}} \tilde{\boldsymbol{\mathcal{A}}_0}^{-1} [\boldsymbol{F}_1], (\tilde{\boldsymbol{\mathcal{A}}_0}^{-1})^* [\boldsymbol{\Psi}_i] \right\rangle_{\mathcal{H}} 
   + \mathcal{O}(\delta\omega^2\ln\omega+ \omega^4\ln\omega + \omega^4).
\end{align*}
We now analyze each term separately. For the first term:
\begin{align*}
\left\langle \tilde{\boldsymbol{\mathcal{A}}_0}^{-1} [\boldsymbol{F}_1], \boldsymbol{\Psi}_i \right\rangle_{\mathcal{H}} 
&= \left\langle \begin{pmatrix} \boldsymbol{u}^{\mathrm{in}} \\ \boldsymbol 0 \end{pmatrix}, (\tilde{\boldsymbol{\mathcal{A}}_0}^{-1})^* [\boldsymbol{\Psi}_i] \right\rangle_{\mathcal{H}} \\
&= \left\langle \begin{pmatrix} \boldsymbol{u}^{\mathrm{in}} \\ \boldsymbol 0 \end{pmatrix}, \begin{pmatrix}\boldsymbol 0 \\ \boldsymbol{f}^{(i)} \end{pmatrix} \right\rangle_{\mathcal{H}} \\
&= 0.
\end{align*}
The second term requires more careful analysis. Using the asymptotic expansion of $\boldsymbol{\mathcal{B}}$ and Lemma \ref{A0d:inver}:
\begin{align*}
&\left\langle \boldsymbol{\mathcal{B}} \tilde{\boldsymbol{\mathcal{A}}_0}^{-1} [\boldsymbol{F}_1], (\tilde{\boldsymbol{\mathcal{A}}_0}^{-1})^* [\boldsymbol{\Psi}_i] \right\rangle_{\mathcal{H}} \\
=& \left\langle \left( \omega^2 \ln \omega \boldsymbol{\mathcal{A}}_{1,1;0} + \omega^2 \boldsymbol{\mathcal{A}}_{1,2;0} + \delta \boldsymbol{\mathcal{A}}_{0,0;1} \right) \left[ \tilde{\boldsymbol{\mathcal{A}}_0}^{-1} \begin{pmatrix} \boldsymbol{d} \\ \boldsymbol 0 \end{pmatrix} + \mathcal{O}(\omega) \right], \boldsymbol{\Phi}_i \right\rangle_{\mathcal{H}} \\
=& \omega^2 \ln \omega \left\langle   \tilde{\boldsymbol{\mathcal{A}}_0}^{-1} \begin{pmatrix} \boldsymbol{d} \\ \boldsymbol 0 \end{pmatrix}, 
\boldsymbol{\mathcal{A}}^*_{1,1;0} \left[\boldsymbol{\Phi}_i\right] \right\rangle_{\mathcal{H}} 
   + \omega^2 \left\langle   \tilde{\boldsymbol{\mathcal{A}}_0}^{-1} \begin{pmatrix} \boldsymbol{d} \\ \boldsymbol 0 \end{pmatrix}, \boldsymbol{\mathcal{A}}^*_{1,2;0} \left[\boldsymbol{\Phi}_i\right] \right\rangle_{\mathcal{H}} \\
&+\delta \left\langle   \tilde{\boldsymbol{\mathcal{A}}_0}^{-1} \begin{pmatrix} \boldsymbol{d} \\ \boldsymbol 0 \end{pmatrix}, \boldsymbol{\mathcal{A}}^*_{0,0;1}\left[\boldsymbol{\Phi}_i \right] \right\rangle_{\mathcal{H}} 
   + \mathcal{O}(\omega^3 \ln \omega + \omega^3+\omega\delta).
\end{align*}
Applying Lemma \ref{A0d:inver} to evaluate these inner products:
\begin{align*}
&\left\langle \tilde{\boldsymbol{\mathcal{A}}_0}^{-1} \begin{pmatrix} \boldsymbol{d} \\ \boldsymbol 0 \end{pmatrix}, \boldsymbol{\mathcal{A}}^*_{1,1;0} \left[\boldsymbol{\Phi}_i \right] \right\rangle_{\mathcal{H}} = \rho \tau^2 \eta (d_1 Q_{1i} + d_2 Q_{2i}), \\
&\left\langle  \tilde{\boldsymbol{\mathcal{A}}_0}^{-1} \begin{pmatrix} \boldsymbol{d} \\ \boldsymbol 0 \end{pmatrix}, \boldsymbol{\mathcal{A}}^*_{1,2;0}\left[\boldsymbol{\Phi}_i\right] \right\rangle_{\mathcal{H}} = \rho \tau^2 \eta \left( \ln(\sqrt{\rho} \tau) (d_1 Q_{1i} + d_2 Q_{2i}) + d_1 P_{1i} + d_2 P_{2i} \right), \\
&\left\langle  \tilde{\boldsymbol{\mathcal{A}}_0}^{-1} \begin{pmatrix} \boldsymbol{d} \\ \boldsymbol 0 \end{pmatrix}, \boldsymbol{\mathcal{A}}^*_{0,0;1} \left[ \boldsymbol{\Phi}_i\right] \right\rangle_{\mathcal{H}} = -\tilde{\eta} (d_1 \delta_{1i} + d_2 \delta_{2i}).
\end{align*}
Combining all terms and using the diagonal structure of $\boldsymbol{Q}$ and $\boldsymbol{P}$ (from Lemma \ref{QP}), we obtain:
\begin{align*}
\left\langle \left( \tilde{\boldsymbol{\mathcal{A}}_0} + \boldsymbol{\mathcal{B}} \right)^{-1} [\boldsymbol{F}], \boldsymbol{\Psi}_i \right\rangle_{\mathcal{H}}  =
\begin{cases}
 -\omega^2 \ln \omega \, \rho \tau^2 \eta d_i Q_{ii} - \omega^2 \rho \tau^2 \eta \left( \ln(\sqrt{\rho} \tau) d_i Q_{ii} + d_i P_{ii} \right) \\
\quad +\, \delta \tilde{\eta} d_i \delta_{ij} + \mathcal{O}(\omega^3 \ln \omega + \omega^3 + \omega \delta), \quad& \text{for } i = 1, 2, \\
\mathcal{O}(\omega^3 \ln \omega + \omega^3 + \omega \delta), \quad& \text{for } i = 3,
\end{cases}
\end{align*}
which is the desired result \eqref{dj:equation}.

\end{proof}

\begin{theo}\label{theo:uD}
Let $D$ be a disk of radius $R$. For $\omega,\delta,\epsilon\in \mathbb{R}^+\rightarrow0$, the solution $\boldsymbol u_D$ of \eqref{Lame} located in $D$ admits  the following asymptotic expansion
\begin{align*}
\boldsymbol u_D=\sum^3_{i=1}\xi_i\boldsymbol f^{(i)}\left(1+o(1)\right)
\end{align*}
with $\boldsymbol f^{(i)}$ given by \eqref{f}. The coefficients $\xi_i,i=1,2,3,$ satisfy the following estimates
\begin{align*}
\xi_i=\left\{ \begin{aligned}&\mathcal{O}(1),&& |\omega^2\ln \omega|\ll\epsilon,\\
&\mathcal{O}(|\ln \omega|),&& |\omega^2\ln \omega|=\mathcal{O}(\epsilon),\\
&o(1),&& \epsilon\ll|\omega^2\ln \omega|\ll 1.
\end{aligned}\right.
\end{align*}

\end{theo}

\begin{proof}
According to the solution formulated in \eqref{u}, the investigation of system \eqref{Lame} reduces to analyzing the boundary integral equation  \eqref{AF:equation}. As indicated by \eqref{AB:asy}, this equation is equivalent to
\begin{align}\label{ABF:equa}
\left(\tilde{\boldsymbol{\mathcal{A}}_0} + \boldsymbol{\mathcal{B}} - \boldsymbol{\mathcal{M}} \right) [\boldsymbol{\Psi_F}](\boldsymbol x) = \boldsymbol{F}(\boldsymbol{x}), \quad \boldsymbol{x} \in \partial D.
\end{align}
We decompose the unknown $\boldsymbol{\Psi_F}$ orthogonally as
\begin{align*}
\boldsymbol{\Psi_F}=\tilde{\boldsymbol{\Psi}}_{\boldsymbol F}+\tilde{\boldsymbol{\Psi}}^{\bot}_{\boldsymbol F} \quad\text{with}\quad \langle\tilde{\boldsymbol{\Psi}}_{\boldsymbol F},\tilde{\boldsymbol{\Psi}}^{\bot}_{\boldsymbol F}\rangle=0,
\end{align*}
where $\tilde{\boldsymbol{\Psi}}_{\boldsymbol F}\in \operatorname{ker}{\boldsymbol{\mathcal{A}}_0}$. Then, $\tilde{\boldsymbol{\Psi}}^{\bot}_{\boldsymbol F}$ can be uniquely determined.
 We express the kernel component as a linear combination of the basis functions
\begin{align}\label{F1}
\tilde{\boldsymbol{\Psi}}_{\boldsymbol F}=\sum^3_{i=1}\alpha_i\boldsymbol{\Psi}_i,
\end{align}
where the coefficients $\alpha_i$ are to be determined.

Applying the operator $(\tilde{\boldsymbol{\mathcal{A}}_0}+\boldsymbol{\mathcal{B}})^{-1}$ to both sides of \eqref{ABF:equa},  yields 
\begin{align*}
\left(\boldsymbol{\mathcal{I}}-(\tilde{\boldsymbol{\mathcal{A}}_0}+\boldsymbol{\mathcal{B}})^{-1}\boldsymbol{\mathcal{M}}\right)[\boldsymbol{\Psi_F}]=(\tilde{\boldsymbol{\mathcal{A}}_0}+\boldsymbol{\mathcal{B}})^{-1}[\boldsymbol F].
\end{align*}
Notice that
\begin{align*}
\boldsymbol{\mathcal{M}}[\tilde{\boldsymbol{\Psi}}_{\boldsymbol F}+\tilde{\boldsymbol{\Psi}}^{\bot}_{\boldsymbol F}]=\sum^3_{i=1}\alpha_i\boldsymbol{\Phi}_i\quad \text{and}\quad \tilde{\boldsymbol{\mathcal{A}}_0}^{-1}\left[\sum^3_{i=1}\alpha_i\boldsymbol {\Phi}_i\right]=\tilde{\boldsymbol{\Psi}}_{\boldsymbol F}.
\end{align*}
Using the Neumann series expansion and the asymptotic expansion \eqref{AB:asy}, we derive
\begin{align}\label{Psi:bot}
&\tilde{\boldsymbol{\Psi}}^{\bot}_{\boldsymbol F}+\tilde{\boldsymbol{\mathcal{A}}_0}^{-1}\left(\omega^2\ln\omega\boldsymbol{\mathcal{A}}_{1,1;0}+\omega^2\boldsymbol{\mathcal{A}}_{1,2;0}+\delta\boldsymbol{\mathcal{A}}_{0,0;1}\right)[\tilde{\boldsymbol{\Psi}}_{\boldsymbol F}]+\mathcal{O}(\delta\omega^2\ln\omega+\omega^4\ln\omega+\omega^4)\nonumber\\
=&(\tilde{\boldsymbol{\mathcal{A}}_0}+\boldsymbol{\mathcal{B}})^{-1}[\boldsymbol F].
\end{align}
Taking the inner product of both sides of \eqref{Psi:bot} with $\boldsymbol{\Psi}_i$ gives
\begin{align*}
&\left\langle(\tilde{\boldsymbol{\mathcal{A}}_0}+\boldsymbol{\mathcal{B}})^{-1}[\boldsymbol F],\boldsymbol{\Psi}_i \right\rangle_{\mathcal{H}}\nonumber\\
=&\left\langle \omega^2 \ln \omega \boldsymbol{\mathcal{A}}_{1,1;0} \left[\sum^3_{j=1} \alpha_j \boldsymbol{\Psi}_j \right], \boldsymbol{\Phi}_i \right\rangle_{\mathcal{H}} + \left\langle \omega^2 \boldsymbol{\mathcal{A}}_{1,2;0} \left[\sum^3_{j=1} \alpha_j \boldsymbol{\Psi}_j \right], \Phi_i \right\rangle_{\mathcal{H}} \nonumber\\
&+\delta \left\langle \boldsymbol{\mathcal{A}}_{0,0;1} \left[ \sum^3_{j=1} \alpha_j \boldsymbol{\Psi}_j \right], \boldsymbol{\Phi}_i \right\rangle_{\mathcal{H}} +\mathcal{O}(\delta \omega^2 \ln \omega + \omega^4 \ln \omega+\omega^4).
\end{align*}

For $i=1,2,3$, leveraging the diagonality of the matrices $\boldsymbol{P}$ and $\boldsymbol{Q}$ (from Lemma \ref{QP}), this simplifies to
\begin{align}\label{bj:equation}
\left\langle (\tilde{\boldsymbol{\mathcal{A}}_0} + \boldsymbol{\mathcal{B}})^{-1}[\boldsymbol{F}], \boldsymbol{\Psi}_i \right\rangle_{\mathcal{H}}
&= \omega^2 \ln \omega \, \rho \tau^2 \alpha_i b_i Q_{ii} + \omega^2 \rho \tau^2 \ln(\sqrt{\rho} \tau) \alpha_i b_i Q_{ii} \nonumber \\
&\quad + \omega^2 \rho \tau^2 \alpha_i b_i P_{ii} - \delta \alpha_i b_i a_i + \mathcal{O}(\delta \omega^2 \ln \omega + \omega^4 \ln \omega + \omega^4).
\end{align}
 Combining \eqref{dj:equation} and \eqref{bj:equation}, we obtain for $i=1,2$ that
 \begin{align}\label{abd:estimate}
& \omega^2 \ln \omega \, \rho \tau^2 (\alpha_i b_i + \eta d_i) Q_{ii} + \omega^2 \rho \tau^2 \ln (\sqrt{\rho}\tau) (\alpha_i b_i + \eta d_i) Q_{ii} \nonumber \\
&+ \omega^2 \rho \tau^2 (\alpha_i b_i + \eta d_i) P_{ii} - \delta (\alpha_i b_i + \eta d_i) + \mathcal{O}(\delta \omega^2 \ln \omega + \omega^4 \ln \omega + \omega^4) \nonumber \\
= &\delta \left( \alpha_i b_i (a_i - 1) + (\tilde{\eta} - \eta) d_i \right) + \mathcal{O}(\omega^3 \ln \omega + \omega^3 + \omega \delta).
\end{align}
It is worth pointing out that, according to the expression of $P_{ii}$ in \eqref{p1} and \eqref{p2}, the order corresponding to $\mathcal{O}(\omega^2)$ in the left-hand part of \eqref{abd:estimate} carries imaginary contributions.

In addition, it holds from  \eqref{Psi:bot} that
\begin{align}\label{F2}
\tilde{\boldsymbol{\Psi}}^{\bot}_{\boldsymbol F}&=(\tilde{\boldsymbol{\mathcal{A}}_0}+\boldsymbol{\mathcal{B}})^{-1}[\boldsymbol F]+\mathcal{O}(\omega^2\ln\omega+\omega^2+\delta)\nonumber\\
&=\tilde{\boldsymbol{\mathcal{A}}_0}^{-1}\begin{pmatrix}\boldsymbol d\\ \boldsymbol 0\end{pmatrix}+\mathcal{O}(\omega+\delta).
\end{align}
Then, according to Lemma \ref{A0d:inver} and Lemma \ref{SK:asy}, we obtain
\begin{align}\label{uD}
\boldsymbol u_D&=\widetilde{\boldsymbol{\mathcal{S}}}_D^{\sqrt{\rho}\tau\omega}\left[\sum^3_{j=1}\alpha_jb_j\boldsymbol f^{(j)}+\eta d_1\boldsymbol f^{(1)}+\eta d_2\boldsymbol f^{(2)}+\mathcal{O}(\omega+\delta)\right]\nonumber\\
&=\sum^2_{j=1}(\alpha_jb_j+\eta d_j)\widetilde{\boldsymbol{\mathcal{S}}}_D[\boldsymbol f^{(j)}]+\alpha_3b_3\widetilde{\boldsymbol{\mathcal{S}}}_D[\boldsymbol f^{(3)}]+\sqrt{2\pi R}\beta_{\sqrt{\rho}\tau\omega}
\begin{pmatrix}\alpha_1b_1+\eta d_1\\ \alpha_2b_2+\eta d_2\end{pmatrix}+\mathcal{O}(\omega+\delta).
\end{align}
It follows from Lemma \ref{S:properties} and \eqref{uD} that
\begin{align*}
\boldsymbol u_D=\sum^3_{j=1}\xi_j\boldsymbol f^{(j)}+\mathcal{O}(\omega+\delta),
\end{align*}
where
\begin{align*}
\left\{ \begin{aligned}
&\xi_1=2\pi R\beta_{\sqrt{\rho}\tau\omega}(\alpha_1b_1+\eta d_1),\\
&\xi_2=2\pi R\beta_{\sqrt{\rho}\tau\omega}(\alpha_2b_2+\eta d_2),\\
&\xi_3=-\frac{\tau_1}{2}R\alpha_3b_3,
\end{aligned}\right.\quad \text{in Case 1}
\end{align*}
or
\begin{align*}
\left\{ \begin{aligned}
&\xi_1=(\alpha_1b_1+\eta d_1)\left(\tau_1 R\ln R-\frac{\tau_2}{2} R+2\pi R\beta_{\sqrt{\rho}\tau\omega}\right),\\
&\xi_2=(\alpha_2b_2+\eta d_2)\left(\tau_1 R\ln R-\frac{\tau_2}{2} R+2\pi R\beta_{\sqrt{\rho}\tau\omega}\right),\\
&\xi_3=-\frac{\tau_1}{2}R\alpha_3b_3.
\end{aligned}\right.\quad \text{in Case 2}
\end{align*}

Based on \eqref{dj:equation}, \eqref{bj:equation} and \eqref{abd:estimate}, we can obtain the following estimates associated with $\xi_i$ for $\omega$ located in different regimes.

When $|\omega^2\ln\omega|\ll\epsilon=\sqrt{\frac{\delta}{\tau^2}}$, one has
$$\xi_1=\xi_2=\xi_3=\mathcal{O}\left(\frac{\delta}{\delta}\right)=\mathcal{O}(1).$$

When $|\omega^2\ln\omega|=\mathcal{O}(\epsilon)$, one has
\begin{align*}
\xi_1=\xi_2=\xi_3&=\mathcal{O}\left(\frac{\delta}{\omega^2}\right)=\mathcal{O}\left(\frac{|\omega^2\ln\omega|}{\omega^2}\right)=\mathcal{O}(|\ln \omega|).
\end{align*}

When $\epsilon\ll|\omega^2\ln\omega|\ll 1$, one has
$$\xi_1=\xi_2=\xi_3=\mathcal{O}\left(\frac{\delta}{|\omega^2\ln \omega|}\right)=o(1).$$

We complete the proof.
\end{proof}


\subsection{Far-field patterns}
This section focuses on the far-field patterns of the scattered field in
the exterior domain, including both longitudinal (denoted by $\boldsymbol{u}_{p,\infty}(\hat{\boldsymbol x})$) and transverse (denoted by $\boldsymbol{u}_{s,\infty}(\hat{\boldsymbol x})$) far-field patterns. Here $\hat{\boldsymbol x}=\frac{\boldsymbol x}{|\boldsymbol x|}$. Clearly, these far-field patterns are defined on the unit circle in $\mathbb{R}^2$.
We begin by presenting an important property related to the Kupradze matrix $\boldsymbol{G}^{k}(\boldsymbol x-\boldsymbol y)$, as detailed in \cite{Sevroglou2005Inverse}.
\begin{lemm}\label{le:fun solu far}
As $|\boldsymbol x|\rightarrow +\infty$ and $k\rightarrow0$, it holds for $\boldsymbol{y}\in\partial D$ that

\begin{align*}
\boldsymbol{G}^{k}(\boldsymbol{x}-\boldsymbol{y}) = 
&-\frac{\mathrm{i}+1}{4(\lambda+2\mu)^{3/4}\sqrt{\pi k}} \,
\widehat{\boldsymbol{x}}\widehat{\boldsymbol{x}}^\top \,
e^{-\mathrm{i} \frac{k}{\sqrt{\lambda+2\mu}}\widehat{\boldsymbol{x}}\cdot \boldsymbol{y}} \,
\frac{e^{\mathrm{i}\frac{k}{\sqrt{\lambda+2\mu}}|\boldsymbol{x}|}}{\sqrt{|\boldsymbol{x}|}} \\
&-\frac{\mathrm{i}+1}{4\mu^{3/4}\sqrt{\pi k}} \,
(\mathbf{I}-\widehat{\boldsymbol{x}}\widehat{\boldsymbol{x}}^\top) \,
e^{-\mathrm{i} \frac{k}{\sqrt{\mu}}\widehat{\boldsymbol{x}}\cdot \boldsymbol{y}} \,
\frac{e^{\mathrm{i}\frac{k}{\sqrt{\mu}}|\boldsymbol{x}|}}{\sqrt{|\boldsymbol{x}|}} 
+ \mathcal{O}\left(|\boldsymbol{x}|^{-3/2}\right).
\end{align*}
uniformly with respect to $\hat{\boldsymbol x}$.
\end{lemm}
From Lemma \ref{le:fun solu far}, we directly obtain the corresponding expansion for the specific wavenumber $k = \sqrt{\rho}\omega$:
\begin{align}\label{fun solu far}
\boldsymbol{G}^{\sqrt{\rho}\omega}(\boldsymbol{x}-\boldsymbol{y}) = 
&-\frac{\mathrm{i}+1}{4\rho^{1/4}\mu^{3/4}\sqrt{\pi\omega}} \,
(\mathbf{I}-\widehat{\boldsymbol{x}}\widehat{\boldsymbol{x}}^\top) \,
\frac{e^{\mathrm{i}\frac{\sqrt{\rho}\omega}{\sqrt{\mu}}|\boldsymbol{x}|}}{\sqrt{|\boldsymbol{x}|}} \,
\left(1+\mathcal{O}(\omega)\right) \nonumber\\
&-\frac{\mathrm{i}+1}{4\rho^{1/4}(\lambda+2\mu)^{3/4}\sqrt{\pi \omega}} \,
\widehat{\boldsymbol{x}}\widehat{\boldsymbol{x}}^\top \,
\frac{e^{\mathrm{i}\frac{\sqrt{\rho}\omega}{\sqrt{\lambda+2\mu}}|\boldsymbol{x}|}}{\sqrt{|\boldsymbol{x}|}} \,
\left(1+\mathcal{O}(\omega)\right) 
+ \mathcal{O}\left(|\boldsymbol{x}|^{-3/2}\right).
\end{align}

\begin{theo}
Let $D$ be a disk of  radius $R$. As $\omega,\delta,\epsilon\in\mathbb{R}^+\rightarrow 0$, the solution $\boldsymbol u_{\mathbb{R}^2\backslash D}$ to \eqref{Lame} in the far field admits the expansion
\begin{align}\label{far solution}
\boldsymbol u_{\mathbb{R}^2\backslash D}-\boldsymbol{u}^{\mathrm{in}}(\boldsymbol x)&=\boldsymbol{u}_{p,\infty}(\widehat{\boldsymbol x})\frac{e^{\mathrm{i}\frac{\sqrt{\rho}\omega}{\sqrt{\lambda+2\mu}}|\boldsymbol x|}}{\sqrt{|\boldsymbol x|}}+\boldsymbol{u}_{s,\infty}(\widehat{\boldsymbol x})\frac{e^{\mathrm{i}\frac{\sqrt{\rho}\omega}{\sqrt{\mu}}|\boldsymbol x|}}{\sqrt{|\boldsymbol x|}}+\mathcal{O}\left(\zeta_i|\boldsymbol x|^{-\frac{3}{2}}\right),
\end{align}
where
\begin{align*}
\boldsymbol{u}_{p,\infty}(\widehat{\boldsymbol x}):=-\frac{\mathrm{i}+1}{4\rho^{\frac{1}{4}}(\lambda+2\mu)^{\frac{3}{4}}\sqrt{\pi \omega}}\widehat{\boldsymbol x}\widehat{\boldsymbol x}^\top\sum^2_{i=1}\zeta_i\boldsymbol f^{(i)}\left(1+\mathcal{O}(\omega)\right)+\mathcal{O}\left(\omega^{\frac{1}{2}}+\delta\omega^{-\frac{1}{2}}\right),
\end{align*}
and
\begin{align*}
\boldsymbol{u}_{s,\infty}(\widehat{\boldsymbol x})&:=-\frac{\mathrm{i}+1}{4\rho^{\frac{1}{4}}\mu^{\frac{3}{4}}\sqrt{\pi\omega}}(\mathbf{I}-\widehat{\boldsymbol x}\widehat{\boldsymbol x}^\top)
\sum^2_{i=1}\zeta_i\boldsymbol f^{(i)}\left(1+\mathcal{O}(\omega)\right)+\mathcal{O}\left(\omega^{\frac{1}{2}}+\delta\omega^{-\frac{1}{2}}\right).
\end{align*}
Here, the magnitude $\zeta_i$ is characterized by
\begin{align}\label{estimate}
\zeta_i=\left\{ \begin{aligned}&\mathcal{O}(1),&& |\omega^2\ln \omega|\ll\epsilon,\\
&\mathcal{O}(|\ln \omega|),&& |\omega^2\ln \omega|=\mathcal{O}(\epsilon),\\
&\mathcal{O}\left(\frac{\delta}{|\omega^2\ln \omega|}\right),&& \epsilon\ll|\omega^2\ln \omega|\ll 1.
\end{aligned}\right.
\end{align}
\end{theo}

\begin{proof}
Based on \eqref{F1} and \eqref{F2}, we derive the expression
\begin{align}\label{uR}
\boldsymbol u_{\mathbb{R}^2\backslash D}-\boldsymbol u^{\mathrm{in}}=&\widetilde{\boldsymbol{\mathcal{S}}}_D^{\sqrt{\rho}\omega}\left[\sum^3_{i=1}\alpha_ib_ia_i\boldsymbol f^{(i)}+\sum^2_{i=1}\tilde{\eta}d_i\boldsymbol f^{(i)}+\mathcal{\mathcal{O}}(\omega+\delta)\right]\nonumber\\
=&\int_{\partial D}\boldsymbol{G}^{\sqrt{\rho}\omega}(\boldsymbol x-\boldsymbol y)\left(\sum^3_{i=1}\alpha_ib_ia_i\boldsymbol f^{(i)}+\sum^2_{i=1}\eta a_id_i\boldsymbol f^{(i)}+\mathcal{O}(\omega+\delta)\right)\mathrm{d}\sigma(\boldsymbol y).
\end{align}
Set
\begin{align*}
\zeta_i=2\pi Ra_i(\alpha_ib_i+\eta d_i), \quad\text{for } i=1,2.
\end{align*}
Combining \eqref{fun solu far}, \eqref{uR} and \eqref{f30},
we get \eqref{far solution}. Moreover, the estimate \eqref{estimate} can be obtained from the proof of Theorem \ref{theo:uD}.
\end{proof}

\section{Dilute Phononic Structure}\label{sec:5}

When resonators are arranged periodically, subwavelength bandgaps exist in two-dimensional phononic crystals, as established in  \cite{Ren2025Subwavelength}. This section extends this analysis to the dilute regime, characterized by a distance between adjacent resonators that is significantly larger than the size of an individual resonator, as illustrated in Figure \ref{dilute1}.
\begin{figure}[ht]
\centering
\includegraphics[scale=1]{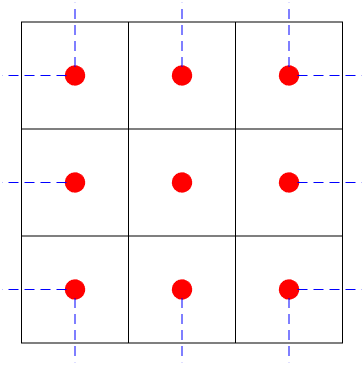}
\caption{Two-dimensional phononic crystal with a dilute structure.}
\label{dilute1}
\end{figure}

Let $\Omega=s D$ with the scaling factor $0<s\ll 1$, where $D$ denotes a disk of radius $R$. Moreover, assume $|\frac{\epsilon}{s^2\ln s}|\ll 1$. For the quasi-momentum variable $\boldsymbol \alpha\in [-\pi,\pi]^2\backslash 0$, define the quasi-periodic single layer potential $\boldsymbol{\mathcal{S}}^{\boldsymbol{\alpha},0}_{\Omega}: L^{2}(\partial \Omega)^2\rightarrow H^{1}(\partial \Omega)^2$ as
\begin{align*}
\boldsymbol{\mathcal{S}}^{\boldsymbol{\alpha},0}_{\Omega}[\boldsymbol\varphi](\boldsymbol x)=\int_{\partial \Omega}\boldsymbol{G}^{\boldsymbol{\alpha},0}(\boldsymbol x-\boldsymbol y)\boldsymbol\varphi(\boldsymbol y)\mathrm{d}\sigma(\boldsymbol y),\quad \boldsymbol x\in\partial \Omega,
\end{align*}
where the quasi-periodic fundamental solution
\begin{align*}
\boldsymbol{G}^{\boldsymbol{\alpha},0}(\boldsymbol x-\boldsymbol y)=\sum_{\boldsymbol{n}\in\mathbb{Z}^{2}}\boldsymbol{G}^{0}(\boldsymbol x-\boldsymbol y-\boldsymbol{n})e^{\mathrm{i}\boldsymbol{n}\cdot\boldsymbol\alpha}
\end{align*}
with $\boldsymbol{G}^{0}(\boldsymbol x)=(G_{ij}^{0}(\boldsymbol x))^2_{i,j=1}$ defined in \eqref{fun solu}. Then, the following lemma can be formulated, with its proof available in Theorem 3.14 of \cite{Ren2025Subwavelength}.

\begin{lemm}\label{phononic}
 For the phononic crystal constructed from the periodic arrangement of resonator $\Omega$, there exists a subwavelength bandgap above the frequency
\begin{align}\label{omega}
\omega^{\boldsymbol \alpha}_{\ast}:=\max_{i=1,2}\max_{\boldsymbol \alpha\in[-\pi,\pi]^2\backslash 0}\sqrt{\frac{\lambda^{\boldsymbol\alpha}_i}{\rho|\Omega|}}\epsilon^{\frac{1}{2}},
\end{align}
where $\lambda^{\boldsymbol\alpha}_i(i=1,2)$ are the eigenvalues of the matrix
\begin{align}\label{matrix}
\boldsymbol {Q^\alpha}:=(Q^{\boldsymbol\alpha}_{ij})^2_{i,j=1}\quad\text{with} \quad Q^{\boldsymbol\alpha}_{ij}=-\int_{\partial \Omega} (\boldsymbol{\mathcal{S}}^{\boldsymbol\alpha,0}_\Omega)^{-1}[\boldsymbol e_i](\boldsymbol x)\cdot\boldsymbol e_j\mathrm{d}\sigma(\boldsymbol x).
\end{align}
Here, $\boldsymbol{e}_i$ denotes the standard unit vector in the $i$-th coordinate direction of $\mathbb{R}^2$.
\end{lemm}

In dilute case, The following theorem further characterizes the frequency $\omega^{\boldsymbol \alpha}_{\ast}$ defined in \eqref{omega}.

\begin{theo}
Let $D$ be a disk of radius $R$ and 
let $\Omega=sD$ with  a scaling factor $0<s\ll1$ such that  $|\frac{\epsilon}{s^2\ln s}|\ll 1$. For $|\boldsymbol\alpha|>c>0$, there exist a subwavelength phononic bandgap slightly above the frequency
\begin{align*}
\omega_{\ast}=\Re\sqrt{\frac{-2t\epsilon}{\rho R s^2}},
\end{align*}
where
\begin{align}\label{t}
t=\frac{1}{\tau_1 R\ln R-\frac{\tau_2}{2} R+2\pi R\beta_{s}}
\end{align}
with $\tau_1,\tau_2$ given by \eqref{parameters} and $\beta_{s}$ given by \eqref{beta}.
\end{theo}

\begin{proof}
Define the matrix-valued function
\begin{align*}
\boldsymbol \Lambda^{\boldsymbol\alpha}=(\Lambda^{\boldsymbol\alpha}_{ij})^2_{i,j=1} \quad \text{with}\quad
  \Lambda^{\boldsymbol\alpha}_{ij}(\boldsymbol x)=G^{\boldsymbol\alpha,0}_{ij}(\boldsymbol x)-G^{0}_{ij}(\boldsymbol x).
\end{align*}
Note that $\boldsymbol \Lambda^{\boldsymbol\alpha}(\boldsymbol x)$ is smooth, and  admits the expansion 
\[\boldsymbol \Lambda^{\boldsymbol\alpha}(\boldsymbol x)=\boldsymbol \Lambda^{\boldsymbol\alpha}(0)+\mathcal{O}(|\boldsymbol x|), \quad \mathrm{~as~}|\boldsymbol x|\to 0.\]
 Let $\boldsymbol x,\boldsymbol y\in \partial D$, so that  $s\boldsymbol x, s\boldsymbol y\in \partial \Omega$. It follows for $\boldsymbol\varphi\in L^{2}(\partial \Omega)^2$ that

\begin{align}\label{S}
\boldsymbol{\mathcal{S}}^{\boldsymbol\alpha,0}_{\Omega}[\boldsymbol\varphi](s \boldsymbol x)=&s\int_{\partial D}\boldsymbol G^{\boldsymbol\alpha,0}(s \boldsymbol x-s \boldsymbol y)\boldsymbol \varphi(s \boldsymbol y)\mathrm{d}\sigma(\boldsymbol y)\nonumber\\
=&s\int_{\partial D}\boldsymbol G^{0}(\boldsymbol x-\boldsymbol y)\widetilde{\boldsymbol \varphi}(\boldsymbol y)\mathrm{d}\sigma(\boldsymbol y)
+s\ln s\frac{\lambda+3\mu}{4\pi \mu(\lambda+2\mu)}\int_{\partial D}\widetilde{\boldsymbol\varphi}(\boldsymbol y)\mathrm{d}\sigma(\boldsymbol y)\nonumber\\
&+s\int_{\partial D}\boldsymbol {\Lambda}^{\boldsymbol\alpha}(0)\widetilde{\boldsymbol\varphi}(\boldsymbol y)\mathrm{d}\sigma(\boldsymbol y)
+s\int_{\partial D}\mathcal{O}(s|\boldsymbol x-\boldsymbol y|)\widetilde{\boldsymbol\varphi}(\boldsymbol y)\mathrm{d}\sigma(\boldsymbol y)\nonumber\\
=&s\hat{\boldsymbol{\mathcal{S}}}^{s}_D[\widetilde{\boldsymbol \varphi}](\boldsymbol x)+s\left(\ln s\frac{\lambda+3\mu}{4\pi \mu(\lambda+2\mu)}-\beta_s\right)\int_{\partial D}\widetilde{\boldsymbol\varphi}(\boldsymbol y)\mathrm{d}\sigma(\boldsymbol y)\nonumber\\
&+s\boldsymbol {\Lambda}^{\boldsymbol\alpha}(0)\int_{\partial D}\widetilde{\boldsymbol\varphi}(\boldsymbol y)\mathrm{d}\sigma(\boldsymbol y)+\mathcal{O}(s^2)\nonumber\\
=&s\hat{\boldsymbol{\mathcal{S}}}^{s}_D\left(\boldsymbol{\mathcal{I}}-(\hat{\boldsymbol{\mathcal{S}}}^{s}_D)^{-1}\boldsymbol{\Pi}\right)[\widetilde{\boldsymbol \varphi}](\boldsymbol x)
\end{align}
where $\widetilde{\boldsymbol\varphi}(\boldsymbol x)=\boldsymbol\varphi(s \boldsymbol x)$ and
\begin{align*}
\boldsymbol{\Pi}[\widetilde{\boldsymbol \varphi}]=\int_{\partial D}\left(\left(\frac{\tau_1}{4\pi}(2\gamma - \pi \mathrm{i} - \ln 4) + \frac{\tau_2}{4\pi} - \frac{\ln (\lambda+2\mu)}{8\pi(\lambda + 2\mu)} -\frac{\ln \mu}{8\pi \mu}\right)\boldsymbol {I}-\boldsymbol \Lambda^{\boldsymbol\alpha}(0)\right)\widetilde{\boldsymbol \varphi}(\boldsymbol y)\mathrm{d}\sigma(\boldsymbol y),
\end{align*}
for $\widetilde{\boldsymbol\varphi}\in L^{2}(\partial D)^2$.
Given that the operator $\hat{\boldsymbol{\mathcal{S}}}^{s}_D$ involves $\ln s$ and $|\ln s|\rightarrow +\infty$ as $s\rightarrow 0$, we can obtain
\begin{align*}
\left\|\left(\hat{\boldsymbol{\mathcal{S}}}^{s}_D\right)^{-1}\boldsymbol{\Pi}\right\|_{\mathcal{L}(L^{2}(\partial D)^2,L^{2}(\partial D)^2)}<1.
\end{align*}
where $\mathcal{L}(L^{2}(\partial D)^2,L^{2}(\partial D)^2)$ denotes the bounded linear operator from $L^{2}(\partial D)^2$ to $L^{2}(\partial D)^2$.
Then, \eqref{S} indicates that
\begin{align}\label{SD}
\left(\boldsymbol{\mathcal{S}}^{\boldsymbol\alpha,0}_{\Omega}\right)^{-1}[\boldsymbol e_i]\approx s^{-1}\left(\hat{\boldsymbol{\mathcal{S}}}^{s}_D\right)^{-1}[\boldsymbol e_i].
\end{align}
We now focus on computing $\left(\hat{\boldsymbol{\mathcal{S}}}^{s}_D\right)^{-1}[\boldsymbol e_i]$. Assume $\widetilde{\boldsymbol\varphi}(\boldsymbol x)\in L^{2}(\partial D)^2$ satisfying $\left(\hat{\boldsymbol{\mathcal{S}}}^{s}_D\right)^{-1}[\widetilde{\boldsymbol\varphi}]=\boldsymbol e_i.$ One has
\begin{align*}
\boldsymbol {\mathcal{S}}_D[\widetilde{\boldsymbol\varphi}]+\beta_s\int_{\partial D}\widetilde{\boldsymbol\varphi}(\boldsymbol y)\mathrm{d}\sigma(\boldsymbol y)=\boldsymbol e_i.
\end{align*}
According to $(\mathrm{i})$ and $(\mathrm{ii})$ in Lemma \ref{S:properties}, it holds that $\widetilde{\boldsymbol\varphi}=t\boldsymbol e_i$ with a constant $t$. Then, we obtain
\begin{align*}
\left(\tau_1 R\ln R-\frac{\tau_2}{2} R\right)t\boldsymbol e_i+2\pi R\beta_st\boldsymbol e_i=\boldsymbol e_i,
\end{align*}
which indicates \eqref{t}.
It holds from \eqref{matrix} and \eqref{SD} that
\begin{align*}
Q^{\boldsymbol\alpha}_{ij}\approx-\Re\int_{\partial D}s^{-1} (\hat{\boldsymbol{\mathcal{S}}}^s_D)^{-1}[\boldsymbol e_i](\boldsymbol x)\cdot\boldsymbol e_j\mathrm{d}\sigma(\boldsymbol x)=-\Re2\pi Rt\delta_{ij},
\end{align*}
which combined with \eqref{omega} yields
\begin{align*}
\omega^{\boldsymbol \alpha}_{\ast}\approx\Re\sqrt{\frac{-2\pi Rt\epsilon}{\rho|\Omega|}}=\Re\sqrt{\frac{-2t\epsilon}{\rho R s^2}}.
\end{align*}
The proof is completed based on Lemma \ref{phononic}.
\end{proof}

\section{Concluding remarks}\label{sec:6}

This paper presents a systematic analysis of wave scattering by a hard elastic inclusion embedded in a soft elastic matrix in two-dimensional space using layer potential techniques. By examining the spectral properties of the associated boundary integral operators, we derive asymptotic expansions for the subwavelength resonant frequencies and characterize their leading-order behavior. Furthermore, we describe the scattering fields both inside and outside the resonator across various regimes of incident frequencies. Finally, extending the results of \cite{Ren2025Subwavelength}, We determine the location of the subwavelength bandgap in a dilute elastic phononic crystal.

As noted in Section \ref{sec:1}, our current analysis is limited to radial configurations due to the specific asymptotic expansions of the two-dimensional fundamental solution and the spectral properties of the boundary integral operators. Generalizing these results to arbitrary geometries will be the subject of future research. We are also particularly interested in developing an effective medium theory for randomly distributed elastic particles, which represents a non-periodic structural configuration.

\section*{Acknowledgment}
The authors express their sincere gratitude to Professor Peijun Li for insightful discussions and valuable suggestions.

%

\begin{thebibliography}{10}

\bibitem{Alu2004Guided}
A.~Alu and N.~Engheta.
\newblock Guided modes in a waveguide filled with a pair of single-negative
  {(SNG)}, double-negative {(DNG)}, and/or double-positive {(DPS)} layers.
\newblock {\em IEEE Trans. Microw. Theory Tech.}, (1):199--210, 2004.

\bibitem{Ammari2019point}
H.~Ammari, D.~P. Challa, A.~P. Choudhury, and M.~Sini.
\newblock The point-interaction approximation for the fields generated by
  contrasted bubbles at arbitrary fixed frequencies.
\newblock {\em J. Differential Equations}, 267(4):2104--2191, 2019.

\bibitem{Ammari2013Spectral}
H.~Ammari, G.~Ciraolo, H.~Kang, H.~Lee, and G.~W. Milton.
\newblock Spectral theory of a {N}eumann-{P}oincar\'{e}-type operator and
  analysis of cloaking due to anomalous localized resonance.
\newblock {\em Arch. Ration. Mech. Anal.}, 208(2):667--692, 2013.

\bibitem{Ammari2017screens}
H.~Ammari, B.~Fitzpatrick, D.~Gontier, H.~Lee, and H.~Zhang.
\newblock A mathematical and numerical framework for bubble meta-screens.
\newblock {\em SIAM J. Appl. Math.}, 77(5):1827--1850, 2017.

\bibitem{AFGLZ2017}
H.~Ammari, B.~Fitzpatrick, D.~Gontier, H.~Lee, and H.~Zhang.
\newblock Sub-wavelength focusing of acoustic waves in bubbly media.
\newblock {\em Proc. A.}, 473(2208):20170469, 17, 2017.

\bibitem{Ammari2018Minnaert}
H.~Ammari, B.~Fitzpatrick, D.~Gontier, H.~Lee, and H.~Zhang.
\newblock Minnaert resonances for acoustic waves in bubbly media.
\newblock {\em Ann. Inst. H. Poincar\'{e} C Anal. Non Lin\'{e}aire},
  35(7):1975--1998, 2018.

\bibitem{AHY2021}
H.~Ammari, E.~O. Hiltunen, and S.~Yu.
\newblock Subwavelength guided modes for acoustic waves in bubbly crystals with
  a line defect.
\newblock {\em J. Eur. Math. Soc. (JEMS)}, 24(7):2279--2313, 2022.

\bibitem{Lay}
H.~Ammari, H.~Kang, and H.~Lee.
\newblock {\em Layer potential techniques in spectral analysis}, volume 153 of
  {\em Mathematical Surveys and Monographs}.
\newblock American Mathematical Society, Providence, RI, 2009.

\bibitem{Libowenmathematical}
H.~Ammari, B.~Li, and J.~Zou.
\newblock Mathematical analysis of electromagnetic scattering by dielectric
  nanoparticles with high refractive indices.
\newblock {\em Trans. Amer. Math. Soc.}, 376(1):39--90, 2023.

\bibitem{ando2018spectral}
K.~Ando, Y.-G. Ji, H.~Kang, K.~Kim, and S.~Yu.
\newblock Spectral properties of the {N}eumann-{P}oincar\'{e} operator and
  cloaking by anomalous localized resonance for the elasto-static system.
\newblock {\em European J. Appl. Math.}, 29(2):189--225, 2018.

\bibitem{Ando2017Spectrum}
K.~Ando, H.~Kang, K.~Kim, and S.~Yu.
\newblock Spectrum of {N}eumann-{P}oincar\'{e} operator on annuli and cloaking
  by anomalous localized resonance for linear elasticity.
\newblock {\em SIAM J. Math. Anal.}, 49(5):4232--4250, 2017.

\bibitem{Blasten2020Localization}
E.~{Bl\aa sten}, H.~Li, H.~Liu, and Y.~Wang.
\newblock Localization and geometrization in plasmon resonances and geometric
  structures of {N}eumann-{P}oincar\'{e} eigenfunctions.
\newblock {\em ESAIM Math. Model. Numer. Anal.}, 54(3):957--976, 2020.

\bibitem{Calvo2012Low}
D.~C. Calvo, A.~L. Thangawng, and C.~N. Layman.
\newblock Low-frequency resonance of an oblate spheroidal cavity in a soft
  elastic medium.
\newblock {\em J. Acoust. Soc. Amer.}, 132(1):EL1--7, 2012.

\bibitem{Casse2010Super}
B.~D.~F. Casse, W.~T. Lu, Y.~J. Huang, E.~Gultepe, L.~Menon, and S.~Sridhar.
\newblock Super-resolution imaging using a three-dimensional metamaterials
  nanolens.
\newblock {\em Appl. Phys. Lett.}, 96(2), 2010.

\bibitem{chang2007spectral}
T.~K. Chang and H.~J. Choe.
\newblock Spectral properties of the layer potentials associated with
  elasticity equations and transmission problems on {L}ipschitz domains.
\newblock {\em J. Math. Anal. Appl.}, 326(1):179--191, 2007.

\bibitem{Chen2025Analysis}
B.~Chen, Y.~Gao, P.~Li, and Y.~Ren.
\newblock Analysis of subwavelength resonances in high contrast elastic media
  by a variational method.
\newblock {\em arXiv preprint arXiv:2501.07315}, 2025.

\bibitem{Chen2024Resonant}
B.~Chen, Y.~Gao, Y.~Li, and H.~Liu.
\newblock Resonant modal approximation of time-domain elastic scattering from
  nano-bubbles in elastic materials.
\newblock {\em Multiscale Model. Simul.}, 22(2):713--751, 2024.

\bibitem{Feppon2022modal}
F.~Feppon and H.~Ammari.
\newblock Modal decompositions and point scatterer approximations near the
  {M}innaert resonance frequencies.
\newblock {\em Stud. Appl. Math.}, 149(1):164--229, 2022.

\bibitem{AAMYB2016}
A.~I. Kuznetsov, A.~E. Miroshnichenko, M.~L. Brongersma, Y.~S. Kivshar, and
  B.~Luk'yanchuk.
\newblock Optically resonant dielectric nanostructures.
\newblock {\em Science}, 354(6314):aag2472, 2016.

\bibitem{Lanoy2015Subwavelength}
M.~Lanoy, R.~Pierrat, F.~Lemoult, M.~Fink, V.~Leroy, and A.~Tourin.
\newblock Subwavelength focusing in bubbly media using broadband time reversal.
\newblock {\em Phys Rev B.}, 91(22):224202, 2015.

\bibitem{Li2016anomalous}
H.~Li and H.~Liu.
\newblock On anomalous localized resonance for the elastostatic system.
\newblock {\em SIAM J. Math. Anal.}, 48(5):3322--3344, 2016.

\bibitem{Li2017three}
H.~Li and H.~Liu.
\newblock On three-dimensional plasmon resonances in elastostatics.
\newblock {\em Ann. Mat. Pura Appl. (4)}, 196(3):1113--1135, 2017.

\bibitem{LZ2024}
H.~Li and J.~Zou.
\newblock Mathematical theory on dipolar resonances of hard inclusions within a
  soft elastic material.
\newblock {\em arXiv preprint arXiv:2310.12861}, 2023.

\bibitem{Liu2005Analytic}
Z.~Liu, C.~T. Chan, and P.~Sheng.
\newblock Analytic model of phononic crystals with local resonances.
\newblock {\em Phys. Rev. B}, 71:014103, Jan 2005.

\bibitem{LZMZYCS2000}
Z.~Liu, X.~Zhang, Y.~Mao, Y.~Zhu, Z.~Yang, C.~T. Chan, and P.~Sheng.
\newblock Locally resonant sonic materials.
\newblock {\em Science}, 289(5485):1734--1736, 2000.

\bibitem{Matlack2016Composite}
K.~H. Matlack, A.~Bauhofer, S.~Kr{\"o}del, A.~Palermo, and C.~Daraio.
\newblock Composite 3d-printed metastructures for low-frequency and broadband
  vibration absorption.
\newblock {\em Proc. Natl. Acad. Sci. U. S. A.}, 113(30):8386--8390, 2016.

\bibitem{Milton2006cloaking}
G.~W. Milton and N.-A.~P. Nicorovici.
\newblock On the cloaking effects associated with anomalous localized
  resonance.
\newblock {\em Proc. R. Soc. Lond. Ser. A Math. Phys. Eng. Sci.},
  462(2074):3027--3059, 2006.

\bibitem{Ren2025Subwavelength}
Y.~Ren, B.~Chen, Y.~Gao, and P.~Li.
\newblock Subwavelength phononic bandgaps in high-contrast elastic media.
\newblock {\em arxiv preprint arxiv:2503.21181}, 2025.

\bibitem{Sevroglou2005Inverse}
V.~Sevroglou.
\newblock The far-field operator for penetrable and absorbing obstacles in 2{D}
  inverse elastic scattering.
\newblock {\em Inverse Problems}, 21(2):717--738, 2005.

\bibitem{Verchota1984Layer}
G.~Verchota.
\newblock Layer potentials and regularity for the {D}irichlet problem for
  {L}aplace's equation in {L}ipschitz domains.
\newblock {\em J. Funct. Anal.}, 59(3):572--611, 1984.

\end{thebibliography}

 \end{document}